\documentclass[11pt]{article}
\usepackage{amsmath}
\usepackage{cases}
\usepackage{mathrsfs}
\usepackage{amsfonts}
\usepackage{amsmath,amsthm,amssymb}
\usepackage{colortbl}

\theoremstyle{plain}
\newtheorem{theorem}{Theorem}[section]
\newtheorem{definition}{Definition}[section]
\newtheorem{lemma}{Lemma}[section]
\newtheorem{remark}{Remark}[section]
\newtheorem{proposition}{Proposition}[section]

\newif \ifLastSection \LastSectionfalse

\numberwithin{equation}{section}

\begin{document}

\title{\bf Large time behavior of solutions to a diffusion approximation radiation hydrodynamics model}

\vskip 2.5cm
\author{
Wenjun Wang$^{1}$\thanks{E-mail: wwj001373@hotmail.com},\ \ \ \
Feng Xie$^{2}$\thanks{E-mail: tzxief@sjtu.edu.cn},\ \ \ \
Xiongfeng Yang$^{2}$\thanks{E-mail: xf-yang@sjtu.edu.cn},
\vspace{2mm}\\
\textit{\small 1. College of Science, University of Shanghai for
Science and Technology,}\\\textit{\small Shanghai 200093, P.R. China}\\
\textit{\small 2. School of Mathematical Sciences, CMA-Shanghai, and MOE-LSC, Shanghai Jiao Tong
University,} \\ \textit{\small Shanghai 200240, P.R. China} \\
 }

\date{}

\maketitle

\textbf{{\bf Abstract:}} This paper concerns with the large time behavior of solutions to a diffusion approximation radiation hydrodynamics model when the initial data is a small perturbation around an equilibrium state. The global-in-time well-posedness of solutions is achieved in Sobolev spaces depending on the Littlewood-Paley decomposition technique together with certain elaborate energy estimates in frequency space. Moreover, the optimal decay rate of the solution is also yielded provided the initial data also satisfy an additional $L^1$ condition. Meanwhile, the similar results of the diffusion approximation system without the thermal conductivity could be also established.


\bigbreak  {\textbf{AMS Subject Classification :} 76N15; 76N10; 74H40

\tableofcontents

\section{Introduction}

\subsection{Background and motivation}

The radiation hydrodynamics system describes the coupling effect between the macroscopic description of the fluid and the statistical character of the massless photons. It finds many applications in modelling the combustion, high-temperature hydrodynamics and gaseous stars in astrophysics etc. The radiative effect is necessary to be included into the hydrodynamics equations for the high-temperature fluid, because the energy and momentum carried by the radiation field are significant in comparison with those carried by the macroscopic fluid. The readers may consult the monographs \cite{LMH,MM} for more details.

Here the macroscopic fluid is described by the compressible Navier-Stokes-Fourier system which represents the conservation of mass, momentum and energy, while the motion of photons is described by the dynamics of the radiation field which is incorporated in a scalar quantity: the radiative intensity $I = I(t,x,\vec \omega,\nu)$. It represents the radiative intensity of the photon which moves in the direction vector $\vec\omega \in \mathbb{S}^{n-1}$ ($\mathbb{S}^{n-1}$ denotes the unit sphere in $\mathbb{R}^n$) with frequency $ \nu \geq 0$ at the position $x$ and time $t$. The time evolution of $I$ is governed by a transport equation with a source term due to the absorbing and the scattering effects of the photons. That is,
\begin{eqnarray}\label{1equationofintencivity}
\frac{1}{\mathcal C} \partial_t I +\vec \omega \cdot \nabla_x I = S,
\end{eqnarray}
where $\mathcal C$ is the light speed and the source term $S$ is given by
\begin{eqnarray*}
S = \sigma_a (B(\nu,\theta)-I) + \sigma_s \Big(\frac{1}{|\mathbb{S}^{n-1}|} \int_{\mathbb{S}^{n-1}} I(\cdot,\vec{\omega})d\vec{\omega}-I \Big).
\end{eqnarray*}
The absorption coefficient $\sigma_a = \sigma_a(\nu,\rho,\theta )$
and the scattering coefficient $\sigma_s=\sigma_s(\nu,\rho,\theta)$ usually depend on the temperature $\theta$ and density $\rho$ of the macroscopic fluid.  $B(\nu,\theta)\geq 0$ is the
equilibrium thermal distribution of radiative intensity, $|\mathbb{S}^{n-1}|$ is the measure of $\mathbb{S}^{n-1}$.
The collective effect of radiation is expressed in 	
terms of integral means (with respect to the variables $\vec{\omega} $ and $\nu$) of quantities depending on $I$.

The coupled system of the compressible Navier-Stokes-Fourier equations with the transport equation above is usually hard for both analysis and numerical simulation.  So some simplified but effective models are proposed. Based on the following two assumptions, the above transport equation could be approximated for the sake of mathematical analysis and computation.
One assumption is the \textit{grey} approximation, it means that the transport coefficients $\sigma_a, \sigma_s$ are assumed to be independent of the frequency $\nu$; Another one is called {\it P1-approximation}, it is assumed that the radiative intensity $I$ could be expanded as a linear function with respect to the angular variable $\vec \omega$, that is,
$$I= I_0 + \vec{\omega} \cdot \vec{I}_1,$$
where $I_0$ and $\vec{I}_1$ are independent of $\vec \omega$ and $\nu$.  Taking the $\{1, \vec\omega\}$-moments of equation (\ref{1equationofintencivity}) and integrating the resultant equation over $\vec\omega \in \mathbb{S}^{n-1}$ and $\nu \in {\mathbb R}^+$, it yields that
\begin{eqnarray*}
\frac{1}{\mathcal{C}} \partial_t I_0 + \frac{1}{n} {\rm div} \vec I_1 = S_E,~~~~
\frac{1}{\mathcal{C}}\partial_t \vec I_1 +\nabla I_0 =\vec S_F.
\end{eqnarray*}
Here the radiation energy and radiation flux are given by
\begin{eqnarray*}
S_E= \frac{1}{\mathcal C} \int_0^{\infty}\int_{\mathbb{S}^{n-1}} S(\cdot,\nu,\vec{\omega})d\vec\omega d\nu, ~~\text{and}~~ \vec S_F= \int_0^{\infty}\int_{\mathbb{S}^{n-1}} \vec{\omega} S(\cdot,\nu,\vec{\omega})d\vec\omega d\nu.
\end{eqnarray*}

 The motion of the macroscopic fluid with the photons is achieved
 though additional extra source terms in the balance of momentum and energy. In the spacial dimension $n=3$, the {\it P1-approximation} radiation hydrodynamics system takes the following form
\begin{eqnarray}\label{1.1}
\left\{ \begin{array}{llllll}
 \partial_t \rho + {\rm div} (\rho \vec u)=0, \\[2mm]
 \partial_t (\rho \vec u) + {\rm div}(\rho \vec u \otimes \vec u)+ \nabla P = {\rm div} \mathbb{T} + \frac{1}{3\mathcal{C}} \mathcal{L}(\sigma_a+\sigma_s) \vec I_1,\\[2mm]
\partial_t (\frac{1}{2}\rho |\vec u|^2+\rho e) +{\rm div}[(\frac{1}{2}\rho |\vec u|^2+\rho e+ P)\vec u ]
 +
 {\rm div} \vec F =
 {\rm div}(\mathbb{T}\cdot\vec u)
-\mathcal{L}\sigma_a(b(\theta)-I_0),\\[2mm]
\frac{1}{\mathcal{C}} \partial_t I_0 + \frac{1}{3} {\rm div} \vec I_1 = \mathcal{L} \sigma_a(b(\theta)-I_0),\\[2mm]
\frac{1}{\mathcal{C}}\partial_t \vec I_1 +\nabla I_0 =-\mathcal{L}(\sigma_a+\sigma_s)\vec I_1,
\end{array}
\right.
\end{eqnarray}
for $(x,t)\in \mathbb{R}^3\times [0,+\infty)$. Here $\mathcal L$, $\sigma_a$ and $\sigma_s$ are positive dimensionless parameters related to the radiation field. The pressure of the macroscopic fluid takes the form of $P= \frac{2}{3} \rho e$ and inner energy $e=\frac{3}{2} R \theta$ with the density $\rho$ and the temperature $\theta$, $\vec u$ is velocity vector. Without loss of generality, we assume the constant $R=1$ for simplicity. $\mathbb{T}$ stands
for the viscous stress tensor determined by Newton's rheological law.
\begin{eqnarray*}
\mathbb{T}= \mu \big(\nabla u+ \nabla^T u\big)+\lambda {\rm div} u {\rm \mathbb{I}}_{3\times 3},
\end{eqnarray*}
where $\mu$ is the shear viscosity coefficient and $\lambda=\zeta-\frac{2}{3}\mu$ with the bulk viscosity coefficient $\zeta\geq 0$. ${\rm \mathbb{I}_{3\times 3}}$ is the $3\times 3$ identity matrix. The heat flux $\vec F= -\kappa \nabla \theta$ satisfies the
Fourier principle with the thermal conductivity $\kappa$. The smooth function $b(\theta)$ is the integral of $B(\nu,\theta)$ with respect to the frequency $\nu$, for example $b(\theta)=\theta^4$ for the case that $B=\frac{2h\nu^3}{C^2}(e^{\frac{k \nu}{h\theta}}-1)^{-1}$ with Planck and Boltzmann constants $h$ and $k$ respectively.

In general, the first order corrector function $\vec I_1$ changes very small with respect to time $t$ for the ``almost" isotropic case.  In this way, we can assume that $\partial_t \vec I_1=0$ in the fifth equation in (\ref{1.1}).  Thus $I_0$ and $\vec I_1$ satisfy the following Fick's principle, that is
\begin{eqnarray*}
 -\nabla I_0 =\mathcal{L}(\sigma_a+\sigma_s)\vec I_1.
\end{eqnarray*}
So, we will obtain the diffusion approximation radiation hydrodynamics system.
\begin{eqnarray}\label{1.2}
\left\{ \begin{array}{llllll}
 \partial_t \rho + {\rm div} (\rho \vec u)=0, \\[2mm]
 \rho \partial_t \vec u + (\rho \vec u \cdot \nabla) \vec u + \nabla P = {\rm div} \mathbb{T} - \frac{1}{3\mathcal{C}}\nabla I_0,\\[2mm]
\frac{3}{2}\rho \partial_t \theta + P {\rm div} \vec u - \kappa \triangle \theta + \mathcal{L}\sigma_a(b(\theta)-I_0)
= (\mathbb{T} \cdot \nabla)\cdot \vec u +\frac{1}{3\mathcal{C}}\vec u \cdot \nabla I_0 -\frac{3}{2}\rho \vec u \cdot \nabla \theta,\\[2mm]
\frac{1}{\mathcal{C}} \partial_t I_0 - \frac{1}{3\mathcal{L}(\sigma_a+\sigma_s)} \triangle I_0 =\mathcal{L} \sigma_a(b(\theta)-I_0).
\end{array}
\right.
\end{eqnarray}
 In this paper we consider the Cauchy problem for the diffusion approximation radiation hydrodynamics system of equations \eqref{1.2} with the following initial data.
\begin{equation}\label{1.3}
(\rho,\vec u,\theta,I_0)(x,t)|_{t=0}
 =
  (\rho_0, \vec u_0, \theta_0, I_0^0)(x)
   \rightarrow (1, \vec 0, 1, b(1)),
    \ \ \ \
    {\rm as}\ |x|\rightarrow +\infty,
\end{equation}
where $(1,\vec 0,1,b(1))$ is the equilibrium state of the system \eqref{1.2}. For simplicity, we assume that the given smooth function $b(\theta)$ satisfies a natural physical assumption of $b^\prime(1)>0$.

Before proceeding, let us review the related known results. System \eqref{1.2} is reduced to the classical non-isentropic compressible Navier-Stokes-Fourier equations if we ignore the radiative effect. It is well-known that the strong dissipative property admits global solutions to the non-isentropic compressible Navier-Stokes-Fourier equations, see \cite{Danchin,De,Matsumura-Nishida} for the global existence results with the small perturbation initial data, \cite{Liu-Wang, XY} for the pointwise asymptotic behaviors of the solutions, and \cite{DM} for the global well-posedness of the equations even without the heat conductivity (i.e. the parameter $\kappa=0$). One also refers to \cite{NS5,KS} and references therein for related results of Navier-Stokes-Fourier equation in exterior domain. In addition, the solutions are proved to blow up in finite time when the initial data has compact support besides it is large in some Sobolev spaces, see \cite{Xin}.

When the radiation effect is taken into account, system \eqref{1.2} can be viewed as the Navier-Stokes-Fourier equations coupled with an parabolic equation with high order nonlinear term with respect to the temperature $\theta$. In the absence of both the viscosity and the heat conductivity, by denoting $q:=-\nabla I_0$ and ignoring $\partial_tI_0$, system \eqref{1.2} can be formulated as the following well-known system of the radiation hydrodynamics.
\begin{eqnarray}\label{1.4}
\left\{ \begin{array}{llllll}
 \partial_t \rho + {\rm div} (\rho \vec u)=0, \\[2mm]
 \rho \partial_t \vec u + (\rho \vec u \cdot \nabla) \vec u + \nabla P = 0,\\[2mm]
\frac{3}{2}\rho \partial_t \theta + P {\rm div} \vec u + \frac{1}{3\mathcal{L}(\sigma_a+\sigma_s)} {\rm div} q
= -\frac{3}{2}\rho \vec u \cdot \nabla \theta,\\[2mm]
 - \frac{1}{3\mathcal{L}(\sigma_a+\sigma_s)} \nabla{\rm div}q +\mathcal{L} \sigma_a q+\mathcal{L} \sigma_a \nabla b(\theta)=0.
\end{array}
\right.
\end{eqnarray}
There are many results providing insight into the existence, uniqueness, asymptotic behavior and decay rates of solutions to the model \eqref{1.4} together with the related ``baby model".
\begin{equation}\label{1.5}
\left\{
\begin{array}{lc}
\partial_t u +{\rm div}f(u)+ {\rm div} q=0,\\[2mm]
-\nabla {\rm div}q+q+\nabla u=0.
\end{array}
\right.
\end{equation}
We refer the readers to \cite{DFZ,F,FRX,GZ,GRZ,KN-1,KN-2,KT,L,LCG,RX,RX-MMMAS,RZ,WW,WX,WX1,YZ,ZS} and the references therein. Let us come back to the radiation hydrodynamic systems \eqref{1.1} and \eqref{1.2} again. Danchin and Ducomet studied the following {\it P1-approximation} radiation hydrodynamics model in \cite{Danchin-Ducomet}.
 \begin{equation}\label{1.6}
\left\{\begin{array}{l}
\partial_t \rho+{\rm div}(\rho \vec u)=0,\\[2mm]
\partial_t(\rho \vec u)+{\rm div}(\rho \vec u\otimes \vec u)+ \frac{1}{ (Ma)^2}\nabla P(\rho)
 =\frac{1}{Re}{\rm div}\mathbb{T}
   +
    \frac{1}{3}\mathcal{L}(\sigma_s+\sigma_a) \vec{I}_1,\\[2mm]
\frac{1}{\mathcal{C}}
 \partial_t I_0
  +
   \frac{1}{3}
    {\rm div} \vec{I}_1
     =
      \mathcal{L}\sigma_a(b(\rho)-I_0),\\[2mm]
\frac{1}{\mathcal{C}}
 \partial_t \vec{I}_1
  +
   \nabla I_0
    =
     -
      \mathcal{L}
       (\sigma_a+\sigma_s)
        \vec{I}_1.\\
 \end{array}
        \right.
\end{equation}
Where the global-in-time existence of strong small perturbation solution was established in the critical Besov spaces $\dot B^{n/2}_{2,1}(\mathbb{R}^n)$. Moreover, the global existence of the solution in critical Besov space for radiation hydrodynamics model \eqref{1.1} also has been achieved in \cite{Danchin-Ducomet-2} recently. It is noticed that the large time behavior of the solutions have not been given in both of these two papers. Later, the global well-posedness and the large time behavior of the smooth solution to \eqref{1.6} in Sobolev space also have been studied by the same authors of this paper in \cite{WXY}. It is shown that the interaction between fluid and the radiation effect could produce partial damping effect on the system. Together with the viscosity on the velocity, the existence and the time-decay rate of the solution could be obtained. The damping effect has also been observed in the baby model of the radiating gas \eqref{1.5} by Kawashima in \cite{KN-1}.

As mentioned above, the energy carried by the radiation field usually dominates the total energy for the high-temperature fluids. Consequently, the energy equation should be taken into account for the precise description of the motion of the high-temperature fluids. In this way, system \eqref{1.2} is more important and more interesting in mathematical analysis from the physical point of view. In this paper, we mainly study the global in time existence and the large time behavior of the solution to the Cauchy problem \eqref{1.2}-\eqref{1.3} when the initial data is a small perturbation around the constant equilibria $(1,0,1,b(1))$.

\subsection{Main results}

Now, it is position to state the main results in the following theorems.
\begin{theorem}[Global existence]\label{Theorem 1.1}
Assume that there is a small positive constant $\varepsilon_0$, such that the initial data satisfies
\begin{equation}\label{1.7}
\|(\rho_0-1,\vec u_0,\theta_0-1,I_0^0-b(1))\|_{H^4(\mathbb{R}^3)}
 \leq
  \varepsilon_0.
\end{equation}
 Then the initial value problem \eqref{1.2} and \eqref{1.3} admits a global unique solution $(\rho, \vec u, \theta, I_0)$ in the following sense.
\begin{equation*}
\begin{split}
\rho-1
 \in
 &
  C^0(0,\infty; H^4(\mathbb R^3))\cap C^1(0,\infty; H^3(\mathbb R^3)),\\[2mm]
\vec u, \theta-1, I_0-b(1)
 \in
 &
  C^0(0,\infty; H^4(\mathbb R^3))\cap C^1(0,\infty; H^2(\mathbb R^3)).
\end{split}
\end{equation*}
Moreover, there exists a positive constant $C_0$ such that for any $t\geq 0$, it holds
\begin{equation}\label{Aprioriestimatewiththermal1}
\begin{split}
\|&(\rho-1,\vec u,  \theta-1,I_0-b(1))(t)\|_{H^4(\mathbb R^3)}^2
 \\
 &  +
\int_0^t
\Big(\|[b^\prime(1)(\theta-1)-(I_0-b(1))](\tau)\|_{H^4(\mathbb R^3)}^2 \\
&\hspace{1cm} + \|\nabla \rho(\tau)\|_{H^3(\mathbb R^3)}^2 +
\|\nabla (\vec u,\theta-1,I_0-b(1))(\tau)\|_{H^4(\mathbb R^3)}^2
\Big){\rm d}\tau\\
 \leq
 & C_0
\|(\rho_0-1,\vec u_0,\theta_0-1,I_0^0-b(1))\|_{H^4(\mathbb R^3)}^2.
\end{split}
\end{equation}
\end{theorem}

\begin{remark}
In Theorem 1.1, the appearance of the second line in \eqref{Aprioriestimatewiththermal1} is due to the ``damping'' effect produced by the interaction between fluid and the radiation. This effect also provide some diffusion property in the combination of the unknown functions. In the process of proving, the main difficulty in establishing the global existence of solution for the model \eqref{1.2} comes from deriving the $L_t^1L_x^1$-norm estimate of $\tilde \theta(I_0-b(1))$. The classical energy method dosen't work here. We overcome this difficulty by dividing the solution $(\tilde \rho,\tilde u,\tilde \theta,j_0)$ into three parts, i.e. the low frequency part, the medium frequency part and the high frequency part. By virtue of the ``damping'' effect
and the diffusion structure of the system, the $L_t^1L_x^1$-norm estimate can be bounded in different frequency regions. It is worth pointing out that a suitable combination of the solution in low-frequency regimes enables us to achieve the desired a priori
estimates and hence to establish the global existence of solution.
\end{remark}

Moreover, the large time behavior of the solution is obtained in the following theorem.

\begin{theorem}[Large time behavior of solution]\label{Theorem 1.2}
Under the assumption of Theorem \ref{Theorem 1.1}, suppose further that $\|(\rho_0-1, \vec u_0, \theta_0-1,I_0^0-b(1))\|_{L^1(\mathbb{R}^3)}$ is bounded. Then there exists a positive constant $\tilde C_0$, such that
\begin{equation*}
\begin{split}
\|\nabla^k(\rho-1,\vec u, \theta-1,I_0-b(1))(t)\|_{L^2(\mathbb{R}^3)}
 \leq
 &
  \tilde C_0(1+t)^{-\frac{3}{4}-\frac{k}{2}},\ \ \ \ for\ k=0,1,2,\\[2mm]
\|\nabla^k(\rho-1, \vec u, \theta-1,I_0-b(1))(t)\|_{L^2(\mathbb{R}^3)}
 \leq
 &
  \tilde C_0(1+t)^{-\frac{7}{4}},\ \ \ \ \ \ \ for\ k=3,4.\\[2mm]
\end{split}
\end{equation*}
  Moreover,
\begin{equation*}
\begin{split}
\|\partial_t(\rho-1,\vec u)(t)\|_{L^2(\mathbb{R}^3)}
 \leq&
  \tilde C_0(1+t)^{-\frac{5}{4}},\\[2mm]
\|\partial_t(\theta-1,I_0-b(1))(t)\|_{L^2(\mathbb{R}^3)}
 \leq&
  \tilde C_0(1+t)^{-\frac{3}{4}},\\[2mm]
  \|\nabla^k [b^\prime(1)(\theta-1)-(I_0-b(1))]\|_{L^2(\mathbb{R}^3)}
 \leq&
  \tilde C_0(1+t)^{-\frac{5}{4}-\frac{k}{2}},\ \ \ for\ k=0,1,2.
\end{split}
\end{equation*}
\end{theorem}
\begin{remark}
The decay rate is optimal in the low order derivatives since it is the same as the heat kernel. We believe that the decay rate for the higher order derivatives is not optimal.
\end{remark}

\begin{remark}
From Theorem \ref{Theorem 1.2}, we know that the time decay  rate of the combination $b^\prime(1)(\theta-1)-(I_0-b(1))$ is  $(1+t)^{-\frac{5}{4}-\frac{k}{2}}$,  which is faster than the decay rates of
 both $\theta-1$ and $I_0-b(1)$. It implies that the interaction between the radiation effect and the motion of the macro fluid will produce the cancellation mechanism.
\end{remark}

It is noted that the global well-posedness and large time behavior of solution to the system \eqref{1.2} in the vanishing heat conductivity parameter case, i.e. $\kappa=0$, also can be established. Precisely, we also consider the following system of equations
\begin{eqnarray}\label{1.9}
\left\{ \begin{array}{llllll}
 \partial_t \rho + {\rm div} (\rho \vec u)=0, \\[2mm]
 \rho \partial_t \vec u + (\rho \vec u \cdot \nabla) \vec u + \nabla P = {\rm div} \mathbb{T} - \frac{1}{3\mathcal{C}}\nabla I_0,\\[2mm]
\frac{3}{2}\rho \partial_t \theta + P {\rm div} \vec u + \mathcal{L}\sigma_a(b(\theta)-I_0)
= (\mathbb{T} \cdot \nabla)\cdot \vec u +\frac{1}{3\mathcal{C}}\vec u \cdot \nabla I_0 -\frac{3}{2}\rho \vec u \cdot \nabla \theta,\\[2mm]
\frac{1}{\mathcal{C}} \partial_t I_0 - \frac{1}{3\mathcal{L}(\sigma_a+\sigma_s)} \triangle I_0 =\mathcal{L} \sigma_a(b(\theta)-I_0).
\end{array}
\right.
\end{eqnarray}
The main results are included in the following theorem.

\begin{theorem}[The case of $\kappa=0$]\label{Theorem 1.3}
Assume that there exists a small positive constant $\varepsilon_0$, such that the initial data satisfies
\begin{equation}
\|(\rho_0-1,\vec u_0,\theta_0-1,I_0^0-b(1))\|_{H^4(\mathbb{R}^3)}
 \leq
  \varepsilon_0.
\end{equation}
Then the initial value problem \eqref{1.9} and \eqref{1.3} admits a unique global solution $(\rho, \vec u, \theta, I_0)$, which satisfies
\begin{equation*}
\begin{split}
\rho-1,\theta-1
 \in
 &
  C^0(0,\infty; H^4(\mathbb R^3))\cap C^1(0,\infty; H^3(\mathbb R^3)),\\[2mm]
\vec u,I_0-b(1)
 \in
 &
  C^0(0,\infty; H^4(\mathbb R^3))\cap C^1(0,\infty; H^2(\mathbb R^3)).
\end{split}
\end{equation*}
Moreover, there exists a positive constant $C^\prime_0$ such that for any $t\geq 0$, it holds
\begin{equation*}
\begin{split}
&
\|(\rho-1,\vec u, \theta-1,I_0-b(1))(t)\|_{H^4(\mathbb{R}^3)}^2\\
&
+
\int_0^t
\Big(
\|[b^\prime(1)(\theta-1)-(I_0-b(1))](\tau)\|_{H^4(\mathbb{R}^3)}^2
\\
& \hspace{5mm}  +\|\nabla (\rho,\theta-1)(\tau)\|_{H^3(\mathbb R^3)}^2
+
\|\nabla (\vec u,I_0-b(1))(\tau)\|_{H^4(\mathbb{R}^3)}^2 \Big)
{\rm d}\tau\\
\leq
&
C^\prime_0
\|(\rho_0-1,\vec u_0,\theta_0-1,I_0^0-b(1))\|_{H^4(\mathbb{R}^3)}^2.
\end{split}
\end{equation*}
Furthermore, we assume that $\|(\rho_0-1, \vec u_0, \theta_0-1,I_0^0-b(1))\|_{L^1} $ is bound. Then there exists a positive constant $\tilde C^\prime_0$ such that, for all $t\geq 0$
\begin{equation*}
\begin{split}
\|\nabla^k(\rho-1,\vec u, \theta-1,I_0-b(1))(t)\|_{L^2(\mathbb{R}^3)}
 \leq
 &
  \tilde C^\prime_0(1+t)^{-\frac{3}{4}-\frac{k}{2}},\ \ \ \ {for}\ k=0,1,2,\\[2mm]
\|\nabla^k(\rho-1, \vec u, \theta-1,I_0-b(1))(t)\|_{L^2(\mathbb{R}^3)}
 \leq
 &
  \tilde C^\prime_0(1+t)^{-\frac{7}{4}},\ \ \ \ \ \ \ for\ k=3,4,\\[2mm]
\end{split}
\end{equation*}
and
\begin{equation*}
\begin{split}
\|\partial_t(\rho-1,\vec u)(t)\|_{L^2(\mathbb{R}^3)}
 \leq
 &
  \tilde C^\prime_0(1+t)^{-\frac{5}{4}},\\[2mm]
\|\partial_t(\theta-1,I_0-b(1))(t)\|_{L^2(\mathbb{R}^3)}
 \leq
 &
  \tilde C^\prime_0(1+t)^{-\frac{3}{4}},\\[2mm]
\|\nabla^k [b^\prime(1)(\theta-1)-(I_0-b(1))]\|_{L^2(\mathbb{R}^3)}
 \leq
 &
  \tilde C^\prime_0(1+t)^{-\frac{5}{4}-\frac{k}{2}},\ \ \ for\ k=0,1,2.\end{split}
\end{equation*}
\end{theorem}
\begin{remark}
In order to establish the global existence of solution to the system \eqref{1.9}, some essential observations of the structure from  the system \eqref{1.9} are needed. The main observation is that the dissipative property of $\theta-1$ comes from two aspects: the damping effect of the combination $b^\prime(1)(\theta-1)-(I_0-b(1))$ and the dissipative effect of $I_0-b(1)$. It is different from the work \cite{DM} for Navier-Stokes-Fourier system without the heat conductivity, where the entropy dosen't dissipate so it has no decay-in-time.
\end{remark}

\begin{remark}
 The results in Theorem \ref{Theorem 1.3} show that the decay properties of the solution to \eqref{1.9} are the same as those of \eqref{1.2}. In fact, for the case of $\kappa=0$, the eigenvalues of the matrix $A(\xi)$ in \eqref{3.311} with asymptotic expansion near 0 can be formulated as
\begin{equation*}
\pm i
\rho
\sqrt{\frac{8\gamma+15\mathcal{C}b}{6\gamma+9\mathcal{C}b}}
+
\Big[
\frac{9\gamma}{(6\gamma+9\mathcal{C}b)^2}
+
\frac{2a\gamma(\gamma+3\mathcal{C}b)}{(8\gamma+15\mathcal{C}b)(2\gamma+3\mathcal{C}b)}
+
\frac{\nu}{2}
\Big]
\varrho^2
+
O(\varrho^3),
\end{equation*}
and
\begin{equation*}
\frac{6a\gamma}{15\mathcal{C} b + 8\gamma}\varrho^2+O(\varrho^3),
\ \ \ \
\frac{2\gamma}{3}+\mathcal{C} b
+
\frac{9\mathcal{C} b a-\gamma}{6\gamma+9\mathcal{C} b}
\varrho^2
+
O(\varrho^3).
\end{equation*}
It enables us to follow the same procedure as in the heat conductivity case to obtain the large time behavior of solution to \eqref{1.9}.
\end{remark}

\subsection{Notations}

Throughout this paper, $C$ denotes the generic positive constant depending only on the initial data and physical coefficients but independent of time $t$. For two quantities $a$ and $b$, we will employ the notation $a\lesssim b$ to mean that $a\leq C b$ for a generic positive constant $C$. And $a\sim b$ means $C^{-1}|b|\leq |a| \leq C|b|$.
Moreover, the norms in nonhomogeneous Sobolev spaces $H^s(\mathbb{R}^3)$ and $W^{s,p}(\mathbb{R}^3)$ are denoted by $\|\cdot\|_{H^s}$ and $\|\cdot\|_{W^{s,p}}$ respectively for $s\geq 0$ and $p\geq 1$. $\|\cdot\|_{\dot H^s}$ denotes the norm in homogeneous Sobolev space $\dot H(\mathbb R^3)$. As usual, $(\cdot|\cdot)$ denotes the inner-product in $L^2(\mathbb{R}^3)$. $\nabla^m$ with an integer $m\geq 0$ stands for the usual any spatial derivatives of order $m$. In addition, we apply the Fourier transform to the variable $x\in\mathbb{R}^3$ by $\widehat f(\xi,t)=\int_{\mathbb{R}^3}f(x,t){\rm e}^{-\sqrt{-1}x\cdot \xi}{\rm d}x$ and the inverse Fourier transform to the variable $\xi\in \mathbb{R}^3$ by $(\mathcal{F}^{-1}\widehat f)(x,t)=(2\pi)^{-3}\int_{\mathbb{R}^3}\widehat f(\xi,t){\rm e}^{\sqrt{-1}x\cdot \xi}{\rm d}\xi$.

\vspace{2mm}

The rest of this paper is organized in the following way. In Section 2, we reformulate the problem into a small perturbation frame. In Section 3, we derive a priori estimates in different frequency regimes and prove the global existence of the solution. The large time behavior of the solution is derived in Section 4. In Appendix, we give the definition of homogeneous Besov space and some useful inequalities.


\section{Reformulations}

Now, we linearize the system of equations \eqref{1.2} around the equilibrium $(1,\vec{0},1,b(1))$. Set $\tilde{\rho}= \rho-1$, $\tilde{u}=\vec u$, $\tilde{\theta}=\theta-1$ and $j_0 =I_0-b(1)$, we obtain
  \begin{eqnarray}\label{2.1}
\left\{ \begin{array}{llllll}
 \partial_t \tilde{\rho} + {\rm div}  \tilde{u} =\tilde S^1, \\[2mm]
 \partial_t  \tilde{u}  +  \nabla \tilde{\rho}+ \nabla \tilde{\theta} +\frac{1}{3\mathcal{C}} \nabla j_0 - {\rm div}\mathbb T=\tilde S^2 ,\\[2mm]
\partial_t \tilde{\theta }+ \frac{2}{3} {\rm div} \tilde{u }- \frac{2}{3} \kappa\triangle \tilde{\theta}
+\frac{2}{3}\mathcal{L}\sigma_a(b'(1)\tilde{\theta}-j_0)=\tilde S^3,\\[2mm]
 \partial_t j_0 - \frac{\mathcal{C}}{3\mathcal{L}(\sigma_a+\sigma_s)} \triangle j_0
 -
 \mathcal{C}\mathcal{L} \sigma_a(b'(1)\tilde{\theta}-j_0)=\tilde S^4,
\end{array}
\right.
\end{eqnarray}
where $(\tilde S^1,\tilde S^2,\tilde S^3,\tilde S^4)$ are the nonlinear terms with
\begin{equation*}
\begin{split}
\tilde S^1
=
&
-
{\rm div}(\tilde \rho \tilde u),\\
\tilde S^2
=
&
   -
   (\tilde u\cdot \nabla)\tilde u
   -
   g(\tilde \rho)\nabla \tilde \rho
   -
   h(\tilde \rho)\tilde \theta \nabla \tilde \rho
   +
   g(\tilde \rho){\rm div}\mathbb T
   -
   \frac{1}{3\mathcal{C}}g(\tilde \rho)\nabla j_0,\\
\tilde S^3
=
&
   -\frac{2}{3}\tilde \theta{\rm div}\tilde u
   +
   \frac{2}{3}\kappa g(\tilde \rho)\Delta \tilde \theta
   -
   \frac{2}{3}\mathcal{L}\sigma_a h(\tilde \rho)
   (b(\tilde \theta+1)-b(1)-b^\prime(1)\tilde \theta)\\
&
  -
  \frac{2}{3}\mathcal{L}\sigma_ag(\tilde \rho)(b^\prime(1)\tilde \theta-j_0)
  +
  \frac{2}{3}h(\tilde \rho)(\mathbb T\cdot\nabla)\cdot \tilde u
  +
  \frac{2}{9\mathcal{C}}h(\tilde \rho)
  \tilde u\cdot\nabla j_0
  -
  \tilde u\cdot\nabla\tilde \theta,\\
\tilde S^4
=
&
\mathcal{C}\mathcal{L}\sigma_a(b(\tilde \theta+1)-b(1)-b^\prime(1)\tilde \theta),
\end{split}
\end{equation*}
and
\begin{equation*}
g(\tilde \rho)
 =
  \frac{1}{\tilde \rho+1}-1,
\ \ \ \
h(\tilde \rho)
 =
  \frac{1}{\tilde \rho+1}.
\end{equation*}
The initial data is given accordingly as follows.
\begin{equation}\label{2.2}
\begin{split}
(\tilde \rho, \tilde u,\tilde \theta, j_0)(x,0)
=
&
(\tilde \rho^0, \tilde u^0,\tilde \theta^0, j_0^0)\\
:=
&
(\rho_0-1, \vec u_0,\theta_0-1, I_0^0-b(1))
\rightarrow
  (0,\vec 0, 0, 0),
  \ \
  {\rm as}
   \
   |x|\rightarrow +\infty.
\end{split}
\end{equation}


Next, we will consider the global existence of the solution $(\tilde\rho,\tilde u,\tilde\theta,j_0)$ to (\ref{2.1}) around the steady state $(0, \vec 0, 0, 0)$.
To this end, we define the function space used in this paper.
\begin{equation*}
\begin{split}
X(0,T)=
 \Big\{
 (\tilde\rho,\tilde u,\tilde\theta,j_0)|
 &\  \tilde\rho\in C^0(0,T;H^4(\mathbb{R}^3))\cap C^1(0,T;H^3(\mathbb{R}^3)),\\
 &\  \tilde u,\tilde\theta,j_0\in C^0(0,T;H^4(\mathbb{R}^3))\cap C^1(0,T;H^2(\mathbb{R}^3)),\\
 &\ \nabla \tilde \rho\in L^2(0,T;H^3(\mathbb{R}^3)),\ \nabla \tilde u,\nabla \tilde \theta,\nabla j_0\in L^2(0,T;H^4(\mathbb{R}^3))
 \Big\}.
\end{split}
\end{equation*}

By the standard continuity argument, the global existence of solutions to the Cauchy problem \eqref{2.1} and \eqref{2.2} will be obtained by combining the local existence result with some uniform a priori estimates in Sobolev space $H^4(\mathbb{R}^3)$.

Then, we state the local existence of smooth solutions to the Cauchy problem \eqref{2.1} and \eqref{2.2} as follows.
\begin{proposition}[Local existence]\label{Proposition 2.1}
Let $(\tilde\rho_0,\tilde u_0,\tilde\theta_0,j_0^0)\in H^4(\mathbb R^3)$ such that
\begin{equation*}
\inf\limits_{x\in\mathbb{R}^3}
\{
\tilde \rho^0+1,\tilde\theta^0+1
\}>0.
\end{equation*}
Then there exists a constant $T_0>0$ depending on $\|(\tilde\rho^0,\tilde u^0,\tilde\theta^0,j_0^0)\|_{H^4(\mathbb R^3)}$, such that the initial-value problem \eqref{2.1} admits a unique solution $(\tilde\rho,\tilde u,\tilde\theta,j_0)\in X(0,T_0)$, which satisfies
\begin{equation*}
\inf\limits_{x\in\mathbb{R}^3,0\leq t\leq T_0}
\{
\tilde\rho+1,\tilde\theta+1
\}>0.
\end{equation*}
\end{proposition}

\begin{proof}
The proof can be done by using the standard iteration arguments and fixed point theorem. One also refer to \cite{Danchin-Ducomet-2} (see Section 2.1). We omit the details for simplicity of presentation.
\end{proof}

The following proposition gives some uniform a priori estimates of smooth solutions to \eqref{2.1}, which are the key parts in the proof of Theorem \ref{Theorem 1.1}.

\begin{proposition}[A priori estimates]\label{Proposition 2.2}
Let $(\tilde \rho^0,\tilde u^0,\tilde \theta^0,j_0^0)\in H^4(\mathbb{R}^3)$. Suppose the initial value problem \eqref{2.1} and \eqref{2.2} has a solution $(\tilde \rho,\tilde u,\tilde \theta,j_0)\in X(0,T)$, where $T$ is a positive constant. Then there exist a sufficient small positive constant $\delta$ and a constant $C_1>0$, which are independent of $T$, such that if the initial data
satisfies
\begin{eqnarray}\label{2.3}
\sup\limits_{0\leq t\leq T}\|(\tilde\rho,\tilde u,\tilde\theta,j_0)(t)\|_{H^4(\mathbb R^3)} \leq \delta,
\end{eqnarray}
then for any $t\in [0,T]$,  the following estimate holds true:
\begin{equation}\label{2.4}
\begin{split}
&
\|(\tilde\rho,\tilde u,\tilde\theta,j_0)(t)\|_{H^4(\mathbb R^3)}^2
 +
\int_0^t
\Big(\|(b^\prime(1)\tilde\theta-j_0)(\tau)\|_{H^4(\mathbb R^3)}^2
\\&\hspace{1cm} +\|\nabla \tilde \rho (\tau)\|_{H^3(\mathbb R^3)}^2
+
\|\nabla (\tilde u,\tilde\theta,j_0)(\tau)\|_{H^4(\mathbb R^3)}^2
\Big)
{\rm d}\tau\\
\leq
&
C_1\|(\tilde\rho^0,\tilde u^0,\tilde\theta^0,j_0^0)(t)\|_{H^4(\mathbb R^3)}^2.
\end{split}
\end{equation}
\end{proposition}

\begin{remark}
For the case that $\kappa=0$, the inequality \eqref{2.4} can be replaced by the following estimate.
\begin{equation*}
\begin{split}
&
\|(\tilde\rho,\tilde u,\tilde\theta,j_0)(t)\|_{H^4(\mathbb R^3)}^2
 +
\int_0^t
\Big(\|(b^\prime(1)\tilde\theta-j_0)(\tau)\|_{H^4(\mathbb R^3)}^2
\\&\hspace{1cm}+\|\nabla (\tilde \rho, \tilde\theta) (\tau)\|_{H^3(\mathbb R^3)}^2
+
\|\nabla (\tilde u,j_0)(\tau)\|_{H^4(\mathbb R^3)}^2
\Big)
{\rm d}\tau\\
\leq
&
C_1\|(\tilde\rho^0,\tilde u^0,\tilde\theta^0,j_0^0)(t)\|_{H^4(\mathbb R^3)}^2.
\end{split}
\end{equation*}
\end{remark}



\section{Global existence and uniqueness}

In this section, we will show the global-in-time existence and uniqueness of solution to the Cauchy problem \eqref{2.1} and \eqref{2.2}.
Actually, we only need to derive the key {\it a priori} estimates stated in Proposition \ref{Proposition 2.2}. Firstly, we drop the ``tilde" and rewrite the unknown functions as $(\rho,u,\theta,j_0)$ in \eqref{2.1} for simplicity of the notations. Then the system about $(\rho,u,\theta,j_0)$ can be written as follows.
\begin{eqnarray}\label{3.1}
\left\{ \begin{array}{llllll}
 \partial_t \rho + {\rm div}  u =S^1, \\[2mm]
 \partial_t  u  +  \nabla \rho+ \nabla \theta+\frac{1}{3\mathcal{C}} \nabla j_0 -{\rm div} \mathbb{T}=S^2,\\[2mm]
\partial_t \theta + \frac{2}{3} {\rm div} u- \frac{2\kappa}{3} \triangle \theta + \frac{2}{3}\mathcal{L}\sigma_a(b'(1) \theta-j_0)=S^3,\\[2mm]
 \partial_t j_0 - \frac{\mathcal{C}}{3\mathcal{L}(\sigma_a+\sigma_s)} \triangle j_0 - \mathcal{C}\mathcal{L} \sigma_a(b'(1)\theta-j_0)=S^4,\\[2mm]
(\rho, u, \theta, j_0)(x,t)|_{t=0}
 =
   ( \rho^0, u^0, \theta^0, j_0^0),
\end{array}
\right.
\end{eqnarray}
where the source terms $(S^1,S^2,S^3,S^4):=(\tilde S^1,\tilde S^2,\tilde S^3,\tilde S^4)$ and the initial data $(\rho_0, u_0, \theta_0, j_0^0)
:=(\tilde \rho^0, \tilde u^0,\tilde \theta^0, j_0^0)$. Motiveted by \cite{Danchin-Ducomet}, we adopt the following notations.
\begin{eqnarray}\label{3.2}
\Lambda=: (-\triangle)^{1/2},~~~d:=\Lambda^{-1} {\rm div} u.
\end{eqnarray}
One gets the identity $u=-\Lambda^{-1} \nabla d - \Lambda^{-1} {\rm div}(\Lambda^{-1} {\rm curl} u)$ together with ${\rm div} u=\Lambda d$ and $({\rm curl}u)_i^j=\partial_ju^i-\partial_iu^j$. Setting
$\nu =\lambda +2\mu$, $\gamma= \mathcal{L}\sigma_a b'(1)$, $a=\frac{\mathcal{C}}{3\mathcal{L}(\sigma_a+\sigma_s)}$ and $b=\mathcal{L}\sigma_a$,  we get from \eqref{3.1} that
\begin{eqnarray}\label{3.3}
\left\{ \begin{array}{llllll}
 \partial_t \rho + \Lambda d =S^1, \\[2mm]
 \partial_t  d - \Lambda \rho -\nu \triangle d - \Lambda \theta  -\frac{1}{3\mathcal{C}} \Lambda j_0=D,\\[2mm]
\partial_t \theta + \frac{2}{3} \Lambda d - \frac{2}{3} \kappa \triangle \theta  + \frac{2}{3} \gamma \theta - \frac{2}{3} b j_0=S^3,\\[2mm]
 \partial_t j_0- \mathcal{C}\gamma \theta - a \triangle j_0 + \mathcal{C} b j_0=S^4,\\[2mm]
 (\rho, d,\theta, j_0)(x,t)|_{t=0}
  =
   (\rho^0, d^0,\theta^0, j_0^0)(x),
\end{array}
\right.
\end{eqnarray}
while $\mathcal{P} u=\Lambda^{-1} {\rm curl} u$ where $\mathcal{P}$ is the projection operator on {\rm div}ergence-free vector fields. It satisfies
\begin{eqnarray}\label{3.4}
\left\{
\begin{array}{lc}
 \partial_t \mathcal{P} u - \mu \triangle \mathcal{P} u =\mathcal{P}S^2,\\[2mm]
 \mathcal{P}u(x,t)|_{t=0}
  =
   \mathcal{P} u^0(x),
\end{array}
\right.
\end{eqnarray}
with $D:=\Lambda^{-1}{\rm div}S^2$ and $d^0:=\Lambda^{-1} {\rm div} u^0$.
In fact, to derive the estimates of $u$, we only need to estimate $d$ and $\mathcal{P}u$.
In what follows, we will give the corresponding analysis by the means of the homogeneous Littlewood-Paley decomposition $(\dot \triangle_k)_{k\in \mathbb{Z}}$ (Pls. see Definition \ref{Def 5.1} in Appendix) and some Besov space techniques. For all $k\in \mathbb{Z}$,
applying the homogeneous frequency localized operator $\dot\Delta_k$ to \eqref{3.3} and \eqref{3.4} yields that
\begin{eqnarray}\label{3.5}
\left\{ \begin{array}{llllll}
 \partial_t \rho_k + \Lambda d_k =S^1_k, \\[2mm]
 \partial_t  d_k - \Lambda \rho_k -\nu \triangle d_k - \Lambda \theta_k  -\frac{1}{3\mathcal{C}} \Lambda j_{0,k}=D_k,\\[2mm]
\partial_t \theta_k + \frac{2}{3} \Lambda d_k - \frac{2}{3} \kappa \triangle \theta_k  + \frac{2}{3} \gamma \theta_k - \frac{2}{3} b j_{0,k}=S^3_k,\\[2mm]
 \partial_t j_{0,k}- \mathcal{C}\gamma \theta_k - a \triangle j_{0,k} + \mathcal{C} b j_{0,k}=S^4_k,\\[2mm]
 (\rho_k , d_k ,\theta_k, j_{0,k})(x,t)|_{t=0}
 =
  (\rho_k ^0, d_k^0,\theta_k^0, j_{0,k}^0)(x),
\end{array}
\right.
\end{eqnarray}
as well as
\begin{eqnarray}\label{3.6}
\left\{
\begin{array}{lc}
 \partial_t (\mathcal{P} u)_k - \mu \triangle (\mathcal{P} u)_k =(\mathcal{P} S^2)_k,\\[2mm]
 (\mathcal{P} u)_k (x,t)|_{t=0}
 =
  (\mathcal{P} u_0)_k(x).
\end{array}
\right.
\end{eqnarray}
Here $\rho_k:= \dot\triangle_k \rho$, $d_k: =\dot\triangle_k d$ and so on.

\subsection{Estimates in the high-frequency regimes}

In this subsection, we show the estimates of the solution to the linearized system of \eqref{3.5}-\eqref{3.6}.
From the system \eqref{3.5}, one can get
\begin{equation}\label{3.7}
\begin{split}
& \frac{1}{2}\frac{\rm d}{{\rm d} t}
\int_{\mathbb{R}^3}
\Big( \rho_k^2+ d_k^2+\frac{3}{2}\theta_k^2+ j_{0,k}^2\Big)
{\rm d}x\\
&
+
\int_{\mathbb{R}^3}
\Big(\nu |\Lambda d_k|^2 + \kappa |\Lambda \theta_k|^2+\gamma|\theta_k|^2 + a|\Lambda j_{0,k}|^2+ \mathcal{C} b|j_{0,k}|^2 \Big)
{\rm d}x\\
=
& (b+\mathcal{C}\gamma)\int_{\mathbb{R}^3} j_{0,k} \theta_k{\rm d}x
  +
  \frac{1}{3\mathcal{C}}\int_{\mathbb{R}^3}\Lambda j_{0,k}d_k{\rm d}x\\
&
 +
  \int_{\mathbb{R}^3}
  \Big(
  \rho_k S^1_k
  +
  d_k D_k
  +
  \frac{3}{2}
  \theta_k S^3_k
  +
  j_{0,k} S^4_k
  \Big)
  {\rm d}x.
\end{split}
\end{equation}
At the same time, by the first two equations $\eqref{3.5}_1$-$\eqref{3.5}_2$, we have
\begin{equation}\label{3.8}
\begin{split}
&
 \frac{\rm d}{{\rm d} t}
 \int_{\mathbb{R}^3}
 \left(\nu |\nabla \rho_k|^2 -  \Lambda \rho_k d_k \right)
 {\rm d}x\\
 =
 &
 \int_{\mathbb{R}^3}
  \left(|\Lambda d_k|^2 -|\Lambda \rho_k|^2 -\Lambda\rho_k\Lambda\theta_k -\frac{1}{3\mathcal{C}}\Lambda\rho_k\Lambda j_{0,k}\right)
 {\rm d}x\\
 &
  +
   \int_{\mathbb{R}^3}
  \Big(
  d_k \Lambda S^1_k
  +
  \Lambda \rho_k D_k
  \Big)
 {\rm d}x
 +
 \nu
 \int_{\mathbb{R}^3}
  \nabla \rho_k \cdot \nabla S^1_k
 {\rm d}x.
\end{split}
\end{equation}
Summing up \eqref{3.7} and $\beta_1\times\eqref{3.8}$ with a fixed constant $\beta_1$, one gets
\begin{equation}\label{3.9}
\begin{split}
 &
  \frac{1}{2}
  \frac{\rm d}{{\rm d} t}
  \mathcal{L}_{h,k}(t)
  +
 \beta_1 \|\Lambda \rho_k\|_{L^2}^2
 +
 (\nu-\beta_1) \|\Lambda d_k\|_{L^2}^2 \\[2mm]
 &
 +
 \kappa \|\Lambda \theta_k\|_{L^2}^2
 +
 \gamma \| \theta_k\|_{L^2}^2
 +
 a\|\Lambda j_{0,k}\|_{L^2}^2
 +
 \mathcal{C} b\|j_{0,k}\|_{L^2}^2\\[2mm]
 =
  &
  (b+\mathcal{C}\gamma)\int_{\mathbb{R}^3} j_{0,k} \theta_k{\rm d}x
  +
  \frac{1}{3\mathcal{C}}\int_{\mathbb{R}^3}\Lambda j_{0,k}d_k{\rm d}x
  -
  \beta_1
  \int_{\mathbb{R}^3}
  \left(\Lambda\rho_k\Lambda\theta_k +\frac{1}{3\mathcal{C}}\Lambda\rho_k\Lambda j_{0,k}\right)
 {\rm d}x\\
 &
 +
   \int_{\mathbb{R}^3}
  \Big(
  \rho_k S^1_k
  +
  d_k D_k
  +
  \frac{3}{2}
  \theta_k S^3_k
  +
  j_{0,k} S^4_k
  +
  \beta_1
  d_k \Lambda S^1_k
  +
  \beta_1
  \Lambda \rho_k D_k
  \Big)
  {\rm d}x\\
 &
   +
   \beta_1
   \nu
 \int_{\mathbb{R}^3}
  \nabla \rho_k \cdot \nabla S^1_k
 {\rm d}x,
\end{split}
\end{equation}
with the functional $\mathcal{L}_{h,k}(t)$ being defined as follows.
\begin{equation}\label{3.10}
\mathcal{L}_{h,k}(t)
 :=
  \int_{\mathbb{R}^3}
  \Big(
   \rho_k ^2+\nu \beta_1 |\nabla \rho_k|^2 - \beta_1\Lambda \rho_k d_k+ d_k^2+\frac{3}{2} \theta_k^2+ j_{0,k}^2
  \Big)
  {\rm d}x.
\end{equation}
Inserting the following inequalities
\begin{equation*}
\begin{split}
 (b+\mathcal{C}\gamma)\int_{\mathbb{R}^3}j_{0,k} \theta_k{\rm d}x
  \leq
  &
   \frac{\gamma}{4} \| \theta_k\|_{L^2}^2 + \frac{(b+\mathcal{C}\gamma)^2}{\gamma}\| j_{0,k}\|_{L^2}^2, \\
 \frac{1}{3\mathcal{C}}\int_{\mathbb{R}^3}\Lambda j_{0,k} d_k{\rm d}x
  \leq
  &
   \frac{a}{4} \|\Lambda j_{0,k}\|_{L^2}^2 + \frac{1}{9a\mathcal{C}^2} \|d_k\|_{L^2}^2, \\
   -\beta_1 \int_{\mathbb{R}^3} \Lambda \rho_k \Lambda \theta_k {\rm d}x
    \leq
    &
     \frac{\beta_1}{4}\|\Lambda \rho_k\|_{L^2}^2
     +
     \beta_1 \|\Lambda \theta_k\|_{L^2}^2, \\
   -\beta_1 \frac{1}{3\mathcal{C}}\int_{\mathbb{R}^3}\Lambda \rho_k \Lambda j_{0,k}{\rm d}x
     \leq
     &
      \frac{\beta_1}{4}\|\Lambda \rho_k\|_{L^2}^2
      +
      \frac{\beta_1}{9\mathcal{C}^2}
       \|\Lambda j_{0,k}\|_{L^2}^2,
\end{split}
\end{equation*}
into \eqref{3.9} leads to
\begin{equation}\label{3.11}
\begin{split}
 &
  \frac{1}{2}
  \frac{\rm d}{{\rm d} t}
  \mathcal{L}_{h,k}(t)
  +
 \frac{\beta_1}{2} \|\Lambda \rho_k\|_{L^2}^2
 +
  (\nu-\beta_1) \|\nabla d_k\|_{L^2}^2
  -
  \frac{1}{9a\mathcal{C}^2}\| d_k\|_{L^2}^2
 +
 (\kappa-\beta_1) \|\Lambda \theta_k\|_{L^2}^2\\
 &
 +
 \frac{3}{4}\gamma \| \theta_k\|_{L^2}^2
 +
 (\frac{3a}{4}-\frac{\beta_1}{9\mathcal{C}^2})\|\Lambda j_{0,k}\|_{L^2}^2
 -
 \frac{(b+\mathcal{C}\gamma)^2}{\gamma}\|j_{0,k}\|_{L^2}^2
 +
 \mathcal{C} b\|j_{0,k}\|_{L^2}^2\\
 \leq
  &
      \int_{\mathbb{R}^3}
  \Big(
  \rho_k S^1_k
  +
  d_k D_k
  +
  \frac{3}{2}
  \theta_k S^3_k
  +
  j_{0,k} S^4_k
  +
  \beta_1
  d_k \Lambda S^1_k
  +
  \beta_1
  \Lambda \rho_k D_k
  \Big)
  {\rm d}x\\
  &
   +
   \beta_1
   \nu
 \int_{\mathbb{R}^3}
  \nabla \rho_k \cdot \nabla S^1_k
 {\rm d}x.
\end{split}
\end{equation}
Choosing the positive constant $\beta_1$ to satisfy
\begin{equation}\label{3.12}
\beta_1
 \leq
  \min\{\frac{\kappa}{2},\frac{\nu}{4},\frac{9a\mathcal{C}^2}{4},1\},
\end{equation}
and a positive integer $k_1$ such that
\begin{equation}\label{3.13}
2^{2k_1-3}
 >
 \max\{\frac{1}{9a\nu\mathcal{C}^2},\frac{2(b+\mathcal{C}\gamma)^2}{a\gamma}\}.
\end{equation}
Then, we arrive at the following inequalities, for $k>k_1>0$,
\begin{equation*}
\nu \|\Lambda d_k\|_{L^2}^2 -\frac{1}{9a\mathcal{C}^2}\| d_k\|_{L^2}^2
\geq
\frac{\nu}{2}
\|\Lambda d_k\|_{L^2}^2,
\ \ \ \
 \frac{3a}{4} \|\Lambda j_{0,k}\|_{L^2}^2-\frac{(b+\mathcal{C}\gamma)^2}{\gamma}\|j_{0,k}\|_{L^2}^2
 \geq
  \frac{a}{2} \|\Lambda j_{0,k}\|_{L^2}^2,
\end{equation*}
where we used the Plancherel theorem and the definition of $(\dot\Delta_k)_{k\in\mathbb{Z}}$.
So, it follows that
\begin{equation} \label{3.14}
\begin{split}
 &
  \frac{1}{2}
  \frac{\rm d}{{\rm d} t}
  \mathcal{L}_{h,k}(t)
  +
 \frac{\beta_1}{2} \|\Lambda \rho_k\|_{L^2}^2
 +
 \frac{\nu}{4} \|\Lambda d_k\|_{L^2}^2\\
 &
 +
 \frac{\kappa}{2} \|\Lambda \theta_k\|_{L^2}^2
 +
 \frac{3}{4}\gamma \| \theta_k\|_{L^2}^2
 +
 \frac{a}{4}\|\Lambda j_{0,k}\|_{L^2}^2
 +
 \mathcal{C} b\|j_{0,k}\|_{L^2}^2\\
 \leq
  &
      \int_{\mathbb{R}^3}
  \Big(
  \rho_k S^1_k
  +
  d_k D_k
  +
  \frac{3}{2}
  \theta_k S^3_k
  +
  j_{0,k} S^4_k
  +
  \beta_1
  d_k \Lambda S^1_k
  +
  \beta_1
  \Lambda \rho_k D_k
  -
  \beta_1
  \nu
  \nabla \rho_k \cdot \nabla S^{12}_k
  \Big)
  {\rm d}x\\
 &
   -
 \beta_1
 \nu
 \int_{\mathbb{R}^3}
  \nabla \rho_k \cdot \nabla S^{11}_k
 {\rm d}x,
\end{split}
\end{equation}
where we split $-S^1$ into two parts, that is, $-S^1= S^{11}+ S^{12}$ with $S^{11}= u \cdot \nabla \rho$ and $S^{12}=\rho{\rm div}u$.
Now we estimate the terms on the r.h.s. of \eqref{3.14}.
For the estimates of the first term on the r.h.s. of \eqref{3.14}, by Cauchy-Schwartz inequality, one gets
\begin{eqnarray}\label{3.15}
\begin{split}
&
\int_{\mathbb{R}^3}
  \Big(
  \rho_k S^1_k
  +
  d_k D_k
  +
  \frac{3}{2}
  \theta_k S^3_k
  +
  j_{0,k} S^4_k
  +
  \beta_1
  d_k \Lambda S^1_k
  +
  \beta_1
  \Lambda \rho_k D_k
  -
  \beta_1
  \nu
  \nabla \rho_k \cdot \nabla S^{12}_k
  \Big)
  {\rm d}x\\
 \leq
  &
  \frac{\beta_1}{8}
  \|\Lambda \rho_k\|_{L^2}^2
  +
  \frac{\nu}{8}
  \|\Lambda d_k\|_{L^2}^2
  +
  \frac{\kappa}{8}
  \|\Lambda \theta_k\|_{L^2}^2
  +
  \frac{a}{8}
  \|\Lambda j_{0,k}\|_{L^2}^2
  +
  C \|(S^1_k,D_k,S^3_k,S^4_k,\nabla S^{12}_k)\|_{L^2}^2,
\end{split}
\end{eqnarray}
where the following facts are used. For any integer $k > k_1>0$, it has
\begin{equation*}
\|\rho_k\|_{L^2}\leq 2^{k_1-1}\|\rho_k\|_{L^2}\leq \|\Lambda \rho_k\|_{L^2}\ \ \ \ {\rm and}\ \ \ \
\|\nabla \rho_k\|_{L^2}\sim \|\Lambda \rho_k\|_{L^2},
\end{equation*}
and so on. For the last term on the r.h.s. of \eqref{3.14}, one has
\begin{equation}\label{3.16}
\begin{split}
 -
 \beta_1
 \nu
 \int_{\mathbb{R}^3}
  \nabla \rho_k \cdot \nabla S^{11}_k
 {\rm d}x
 \leq
  &
   \beta_1
   \nu
   \big|\big(((\nabla u)^T\cdot\nabla \rho )_k|\nabla \rho _k\big)\big|
   +
   \beta_1
   \nu
   \big|\big( (u\cdot\nabla)\nabla\rho )_k|\nabla \rho _k\big)\big|\\
 \leq
  &
   \beta_1
   \nu
   \big|\big(((\nabla u)^T\cdot\nabla \rho )_k|\nabla \rho _k\big)\big|
   +
   \beta_1
   \nu
   \big|\big([ \dot\triangle_k,u\cdot \nabla]\nabla\rho |\nabla \rho _k\big)\big|\\
  &
   +
   \beta_1
   \nu
   \big|\big( (u\cdot\nabla)\nabla\rho _k|\nabla \rho _k\big)\big|\\
  \leq
   &
    \frac{\beta_1}{16} \|\nabla \rho _k \|_{L^2}^2
     +
      C \|((\nabla u)^T\cdot\nabla \rho)_k\|_{L^2}^2\\
   &
      +
       C  \big|\big([ \dot\triangle_k,u\cdot \nabla]\nabla\rho |\nabla \rho _k\big)\big|
      +
       C \|{\rm div}  u\|_{L^{\infty}}\|\nabla\rho _k\|_{L^2}^2.
\end{split}
\end{equation}
By using Lemma \ref{Lemma 5.3} and the Young inequality, we get
\begin{equation}\label{3.17}
\begin{split}
 \big|\big([ \dot\triangle_k,u\cdot \nabla]\nabla\rho |\nabla \rho _k\big)\big|
\lesssim
&
 \|\nabla u\|_{L^\infty}
 \|\nabla \rho_k\|_{L^2}^2
 +
 \|\nabla^2 \rho\|_{L^\infty}
 \|\dot\Delta_k u\|_{L^2}
 \|\nabla \rho_k\|_{L^2}\\
&
+
\|\nabla u\|_{L^\infty}
 \|\nabla \rho_k\|_{L^2}
 \sum\limits_{l\geq k-1}
 2^{k-l}
 \|\nabla \rho_l\|_{L^2}\\
\leq
&
\frac{\beta_1}{16}
\|\nabla \rho_k\|_{L^2}^2
+
C
\|\nabla u\|_{L^\infty}
\|\nabla \rho_k\|_{L^2}^2
+
C
\|\nabla ^2\rho\|_{L^\infty}^2
\|u_k\|_{L^2}^2\\
&
+
C
\|\nabla u\|_{L^\infty}^2
\left(
 \sum\limits_{l\geq k-1}
 2^{k-l}
 \|\nabla \rho_l\|_{L^2}\right)^2.
\end{split}
\end{equation}
Putting \eqref{3.15}, \eqref{3.16} and \eqref{3.17} into \eqref{3.14} yields
\begin{equation}\label{3.18}
\begin{split}
   &
   \frac{1}{2}
   \frac{\rm d}{{\rm d}t}\mathcal{L}_{h,k}(t)
   +
 \frac{\beta_1}{4} \|\Lambda \rho_k\|^2
 +
 \frac{\nu}{8} \|\Lambda d_k\|^2\\
 &
 +
 \frac{\kappa}{4} \|\Lambda \theta_k\|^2
 +
 \frac{3}{4}\gamma \| \theta_k\|^2
 +
 \frac{a}{8}\|\Lambda j_{0,k}\|^2
 +
 \mathcal{C} b\|j_{0,k}\|^2\\[1mm]
\lesssim
   &
   \|(S^1_k,D_k,S^3_k,S^4_k,\nabla S^{12}_k )\|_{L^2}^2
   +
   \|((\nabla u)^T\cdot\nabla \rho)_k\|_{L^2}^2
      +
\|\nabla u\|_{L^\infty}
\|\nabla \rho_k\|_{L^2}^2\\[2mm]
   &
+
\|\nabla ^2\rho\|_{L^\infty}^2
\|u_k\|_{L^2}^2
+
\|\nabla u\|_{L^\infty}^2
\left(
 \sum\limits_{l\geq k-1}
 2^{k-l}
 \|\nabla \rho_l\|_{L^2}\right)^2.
\end{split}
\end{equation}

Below, we show all of the first-order derivative estimates of $d_k,~\theta_k$ and $j_{0,k}$. From $\eqref{3.5}_2$-$\eqref{3.5}_4$, we readily get
\begin{eqnarray}\label{3.19}
\begin{split}
 &
  \frac{1}{2}\frac{\rm d}{{\rm d}t} \|(\Lambda d_k, \frac{3}{2}\Lambda \theta_k, \Lambda j_{0,k})\|_{L^2}^2
   +
    \nu \|\Lambda^2 d_k\|_{L^2}^2\\
 &
   +
   \kappa\|\Lambda^2 \theta_k\|_{L^2}^2
  +
   \gamma\|\Lambda \theta_k\|_{L^2}^2
  +
   a \|\Lambda^2 j_{0,k}\|_{L^2}^2
  +
   \mathcal{C}b\|\Lambda j_{0,k}\|_{L^2}^2 \vspace{5pt}\\
=
 &
 (b+\mathcal{C}\gamma)
 (\Lambda j_{0,k}|\Lambda \theta_k)
 +
  \frac{1}{3\mathcal{C}}\big( \Lambda^2 j_{0,k}| \Lambda d_k \big)
  -
  (\triangle \rho_k| \Lambda d_k)\\
 &
  +
  (\Lambda D_k| \Lambda d_k)
  +
  \frac{3}{2}(\Lambda S^3_k|\Lambda \theta_k)
  +
  (\Lambda S^4_k|\Lambda j_{0,k}).
\end{split}
\end{eqnarray}
The Cauchy-Schwarz inequality gives the following inequalities
\begin{equation*}
\begin{split}
 (b+\mathcal{C}\gamma)
 |(\Lambda j_{0,k}|\Lambda \theta_k)|
&
\leq
\frac{(b+\mathcal{C}\gamma)^2}{\gamma}
\|\Lambda j_{0,k}\|_{L^2}^2
+
\frac{\gamma}{4}
\|\Lambda \theta_k\|_{L^2}^2\\
\frac{1}{3\mathcal{C}}
|\big( \Lambda^2 j_{0,k}| \Lambda d_k \big)|
&
 \leq
 \frac{1}{9\nu\mathcal{C}^2}
 \|\Lambda j_{0,k}\|_{L^2}^2
 +
 \frac{\nu}{4}\|\Lambda^2 d_{k}\|_{L^2}^2 ,\vspace{6pt}\\
|(\triangle \rho_k| \Lambda d_k)|
&
 \leq \frac{1}{\nu} \|\Lambda \rho_k\|_{L^2} +\frac{\nu}{4}\|\Lambda^2 d_{k}\|_{L^2}^2 \vspace{6pt}\\
|(\Lambda D_k| \Lambda d_k)|
&
 \leq \frac{1}{\nu}\|D_k\|_{L^2}^2+\frac{\nu}{4}\|\Lambda^2 d_{k}\|_{L^2}^2,\\
 \frac{3}{2}
 |(\Lambda S^3_k|\Lambda \theta_k)|
&
 \leq \frac{9}{4\kappa} \|S^3_k\|_{L^2}^2+ \frac{\kappa}{4} \|\Lambda^2 \theta_k\|_{L^2}^2\\
 |(\Lambda S^4_k|\Lambda j_{0,k})|
&
 \leq \frac{1}{2 a} \|S^4_k\|_{L^2}^2+ \frac{a}{2} \|\Lambda^2 j_{0,k}\|_{L^2}^2.
\end{split}
\end{equation*}
Inserting the estimates above into \eqref{3.19} yields that
\begin{eqnarray}\label{3.20}
\begin{split}
&
\frac{1}{2}
\frac{\rm d}{{\rm d}t}\|(\Lambda d_k,\frac{3}{2}\Lambda\theta_k,\Lambda j_{0,k})\|_{L^2}^2
  +
  \frac{\nu}{4} \|\Lambda^2 d_k\|_{L^2}^2\\
&
  +
  \frac{\kappa}{2}
  \|\Lambda^2 \theta_k\|_{L^2}^2
  +
  \gamma
  \|\Lambda \theta_k\|_{L^2}^2
  +
  \frac{a}{2}
  \|\Lambda^2 j_{0,k}\|_{L^2}^2
  +
   \mathcal{C}b
   \|\Lambda j_{0,k}\|_{L^2}^2\\
\leq
 &
 \frac{1}{\nu}
 \|\Lambda \rho_k\|_{L^2}^2
 +
 \Big(
 \frac{(b+\mathcal{C}\gamma)^2}{\gamma}
 +
 \frac{1}{9\nu\mathcal{C}^2}
 \Big)
 \|\Lambda j_{0,k}\|_{L^2}^2
 +
 \frac{1}{\nu}
 \|D_k\|_{L^2}^2
 +
 \frac{9}{4\kappa}
 \|S^3_k\|_{L^2}^2
 +
 \frac{1}{2a}
 \|S^4_k\|_{L^2}^2.
\end{split}
\end{eqnarray}
For the estimate of $(\mathcal{P}u)_k$, it follows from \eqref{3.6} and the Young inequality that for any integer $k>k_1$,
\begin{equation}\label{3.21}
\begin{split}
&
\frac{1}{2}
\frac{\rm d}{{\rm d}t}
\|\big((\mathcal{P}u)_k,\Lambda (\mathcal{P}u)_k\big)\|_{L^2}^2
    +
    \frac{\mu}{2}
    \|\Lambda^2 (\mathcal{P} u)_k\|_{L^2}^2
 +
    \mu\|\Lambda (\mathcal{P} u)_k\|_{L^2}^2
    -
    \frac{\mu}{2}
    \|(\mathcal{P} u)_k\|_{L^2}^2
\leq
   \frac{1}{\mu}
   \|(\mathcal{P} S^2)_k\|_{L^2}^2.
\end{split}
\end{equation}
Since $k>k_1>0$, we have
\begin{equation*}
    \|(\mathcal{P} u)_k\|_{L^2}^2
\leq
 \|\Lambda (\mathcal{P} u)_k\|_{L^2}^2,
\end{equation*}
which together with \eqref{3.21} imply
\begin{equation}\label{3.22}
\begin{split}
&
\frac{1}{2}
\frac{\rm d}{{\rm d}t}
\|\big((\mathcal{P}u)_k,\Lambda (\mathcal{P}u)_k\big)\|_{L^2}^2
 +
    \frac{\mu}{2}\|\big(\Lambda (\mathcal{P} u)_k,\Lambda^2 (\mathcal{P} u)_k\big)\|_{L^2}^2
\leq
   \frac{1}{\mu}
   \|(\mathcal{P} S^2)_k\|_{L^2}^2.
\end{split}
\end{equation}
Then, adding $\beta_2\times \eqref{3.20}$ and \eqref{3.22} to \eqref{3.18} with $\beta_2>0$ being a suitably small constant, we achieve the estimates of the high frequency part for $k>k_1$ as follows.
 \begin{equation}\label{3.23}
\begin{split}
 &
 \frac{1}{2}
 \frac{\rm d}{{\rm d}t}
 \left\{
\mathcal{L}_{h,k}(t)
+
\beta_2
\|\big(\Lambda d_k,\frac{3}{2}\Lambda \theta_k,\Lambda j_{0,k}\big)(t)\|_{L^2}^2
+
\|\big((\mathcal{P} u)_k,\Lambda (\mathcal{P} u)_k\big)(t)\|_{L^2}^2
 \right\}\\
 &
 +
 \Big(\frac{\beta_1}{4}-\frac{\beta_2}{\nu}\Big)
  \|\Lambda \rho_k(t)\|_{L^2}^2
 +
 \frac{\nu}{8} \|\Lambda d_k(t)\|_{L^2}^2
 +
  \frac{\nu}{4}\beta_2 \|\Lambda^2 d_k(t)\|_{L^2}^2\\[1mm]
 &
 +
 \frac{3}{4}\gamma \| \theta_k(t)\|_{L^2}^2
 +
 \Big(\frac{\kappa}{4}+\beta_2\gamma\Big)
 \|\Lambda \theta_k(t)\|_{L^2}^2
 +
  \frac{\kappa}{2}\beta_2
  \|\Lambda^2 \theta_k(t)\|_{L^2}^2\\
 &
 +
 \mathcal{C} b\|j_{0,k}(t)\|_{L^2}^2
 +
 \Big(
 \frac{a}{8}
 +
 \beta_2\mathcal{C}b
 -
 \frac{\beta_2(b+\mathcal{C}\gamma)^2}{\gamma}
 -
 \frac{\beta_2}{9\nu\mathcal{C}^2}
 \Big)
 \|\Lambda j_{0,k}(t)\|_{L^2}^2\\[1mm]
 &
 +
  \frac{a}{2}\beta_2
  \|\Lambda^2 j_{0,k}(t)\|_{L^2}^2
 +
\frac{\mu}{2}
\|\big(\Lambda (\mathcal{P} u)_k,\Lambda^2 (\mathcal{P} u)_k\big)(t)\|_{L^2}^2\\[1mm]
 \lesssim
 &
 \|( S^1_k,S^2_k,S^3_k,S^4_k,\nabla S^{12}_k)(t)\|_{L^2}^2 {\rm d}\tau
 +
 \|((\nabla u)^T \cdot\nabla\rho)_k(t)\|_{L^2}^2
      +
\|\nabla u\|_{L^\infty}
\|\nabla \rho_k\|_{L^2}^2\\[2mm]
   &
+
\|\nabla ^2\rho\|_{L^\infty}^2
\|u_k\|_{L^2}^2
+
\|\nabla u\|_{L^\infty}^2
\left(
 \sum\limits_{l\geq k-1}
 2^{k-l}
 \|\nabla \rho_l\|_{L^2}\right)^2,
\end{split}
\end{equation}
where we have used the facts that
\begin{equation}\label{3.24}
\|D_k\|_{L^2}
\leq
\|S^2_k\|_{L^2}\ \ \ \ {\rm and}\ \ \ \
\|(\mathcal{P}S^2)_k\|_{L^2}
\leq
\|S^2_k\|_{L^2}.
\end{equation}

Setting
\begin{equation*}
\mathcal{H}_{h,k}(t)
 =
\mathcal{L}_{h,k}(t)
+
\beta_2
\|\big(\Lambda d_k,\frac{3}{2}\Lambda \theta_k,\Lambda j_{0,k}\big)(t)\|_{L^2}^2
+
\|\big((\mathcal{P} u)_k,\Lambda (\mathcal{P} u)_k\big)(t)\|_{L^2}^2,
\end{equation*}
then there exists a positive constant $C_2$ such that
\begin{equation}\label{3.25}
C_2^{-1}
\mathcal{H}_{h,k}(t)
\leq
 \|(\rho_k, u_k, \theta_k, j_{0,k})(t)\|_{L^2}^2
 +
 2^{2k}\|(\rho_k, u_k, \theta_k, j_{0,k})(t)\|_{L^2}^2
\leq
C_2
\mathcal{H}_{h,k}(t),
\end{equation}
where we have used the identity $u_k=-\Lambda^{-1}\nabla d_k-\Lambda^{-1}{\rm div}(\mathcal P u)_k$ and the Bernstein inequality.
From \eqref{3.23}, by letting the constant $\beta_2$ satisfy
\begin{equation}
0<\beta_2\leq
\min\left\{
\frac{\beta_1\nu}{8},
\frac{a}{16}
\Big|
 \frac{(b+\mathcal{C}\gamma)^2}{\gamma}
 +
 \frac{1}{9\nu\mathcal{C}^2}
 -
 \mathcal{C}b
\Big|^{-1}
\right\},
\end{equation}
there is a positive constant $C_3$ depending on $\beta_1,\beta_2,\kappa,\nu,\gamma,a,\mathcal{C}$ and $b$, such that for any $k>k_1$
\begin{equation}\label{3.27}
\begin{split}
   &
   \frac{\rm d}{{\rm d}t}
   \mathcal{H}_{h,k}(t)
 +
 C_3
 \Big\{
 \|(\theta_k,j_{0,k})(t)\|_{L^2}^2
 +
 2^{2k}\|(\rho_k, u_k,\theta_k,j_{0,k})(t)\|_{L^2}^2
 +
 2^{4k}\|(u_k,\theta_k,j_{0,k})(t)\|_{L^2}^2
 \Big\}\\
\lesssim
   &
   \|( S^1_k,S^2_k,S^3_k,S^4_k,\nabla S^{12}_k)(t)\|_{L^2}^2
   +
   \|((\nabla u)^T \cdot\nabla\rho )_k(t)\|_{L^2}^2
      +
\|\nabla u\|_{L^\infty}
\|\nabla \rho_k\|_{L^2}^2\\[2mm]
   &
+
\|\nabla ^2\rho\|_{L^\infty}^2
\|u_k\|_{L^2}^2
+
\|\nabla u\|_{L^\infty}^2
\left(
 \sum\limits_{l\geq k-1}
 2^{k-l}
 \|\nabla \rho_l\|_{L^2}\right)^2.
\end{split}
\end{equation}
Integrating \eqref{3.27} with respect to $t$ over $[0,t]$ and using \eqref{3.25} yield for any $k>k_1$
\begin{equation}\label{3.28}
\begin{split}
 &
 \|(\rho_k, u_k, \theta_k, j_{0,k})(t)\|_{L^2}^2
 +
 2^{2k}\|(\rho_k, u_k, \theta_k, j_{0,k})(t)\|_{L^2}^2\\
 &
 +
 \int_0^t
 \left\{
 \| (\theta_k,j_{0,k})(\tau)\|^2_{L^2}
 +
 2^{2k}\|(\rho_k, u_k, \theta_k, j_{0,k})(\tau)\|^2_{L^2}
 +
 2^{4k}\|(u_k,\theta_k,j_{0,k})(\tau)\|^2_{L^2}
 \right\}{\rm d}\tau \\
 \lesssim
 &
 \|(\rho_k,u_k,\theta_k,j_{0,k})(0)\|_{L^2}^2
 +
 2^{2k}\|(\rho_k, u_k, \theta_k, j_{0,k})(0)\|_{L^2}^2\\
 &
 +
 \int_0^t
 \big(
 \|( S^1_k,S^2_k,S^3_k,S^4_k,\nabla S^{12}_k)(\tau)\|_{L^2}^2
 +
 \|((\nabla u)^T \cdot\nabla\rho)_k(\tau)\|_{L^2}^2
 \big)
 {\rm d}\tau\\
 &
 +
 \int_0^t
\|\nabla u\|_{L^\infty}
\|\nabla \rho_k\|_{L^2}^2
+
\|\nabla ^2\rho\|_{L^\infty}^2
\|u_k\|_{L^2}^2
+
\|\nabla u\|_{L^\infty}^2
\left(
 \sum\limits_{l\geq k-1}
 2^{k-l}
 \|\nabla \rho_l\|_{L^2}\right)^2
 {\rm d}\tau.
\end{split}
\end{equation}

By the weighted $l^2$ summation over $k$ with $k\geq k_1+1$ in \eqref{3.28} and the definition in \eqref{5.1}, we have the following proposition for short wave parts.
\begin{proposition}\label{Proposition 3.3}
Let $s\geq 1$ be a real number. The following inequality holds true for the solution to \eqref{3.1}
\begin{equation}\label{3.29}
\begin{split}
 &
 \|(\rho,  u , \theta, j_0)^S(t)\|_{\dot B_{2,2}^{s-1}}^2
 +
  \|(\rho,  u , \theta, j_0)^S(t)\|_{\dot B_{2,2}^{s}}^2\\
 &
 +
  \int_0^t
  \Big(
 \|(\rho,u,\theta,j_0)^S(\tau) \|_{\dot B_{2,2}^{s}}^2
 +
  \|(u,\theta,j_0) ^S(\tau)\|^2_{\dot B_{2,2}^{s+1}}
 +
 \| (\theta, j_0)^S(\tau) \|_{\dot B_{2,2}^{s-1}}^2
 \Big)
 {\rm d}\tau\\
\lesssim
 &
  \|(\rho, u ,  \theta, j_0)^S(0)\|_{\dot B_{2,2}^{s-1}}^2
  +
  \|(\rho, u ,  \theta, j_0)^S(0)\|_{\dot B_{2,2}^{s}}^2\\
 &
  +
   \int_0^t
   \Big(
   \| (S^{12})^S(\tau)\|_{\dot B^{s}_{2,2}}^2+\|\big(S^1,S^2,S^3,S^4,(\nabla u)^T \cdot\nabla\rho\big)^S(\tau)\|_{\dot B^{s-1}_{2,2}}^2
   \Big)
   {\rm d}\tau\\
 &
 +
 \int_0^t
  \Big(
  \|\nabla^2\rho\|_{L^\infty}^2
  \|u^S\|_{\dot B^{s-1}_{2,2}}^2
  +
  \|\nabla u (\tau)\|_{L^{\infty}}\|\nabla\rho^S(\tau)\|_{\dot B^{s-1}_{2,2}}^2
 \Big) {\rm d}\tau\\
 &
 +
 \int_0^t
  \|\nabla u\|_{L^\infty}^2
  \sum\limits_{k>k_1}2^{2k(s-1)}\left(
 \sum\limits_{l\geq k-1}
 2^{k-l}
 \|\nabla \rho_l\|_{L^2}\right)^2
  {\rm d}\tau.
\end{split}
\end{equation}
\end{proposition}

\subsection{Estimates in the low-frequency regimes}

\subsubsection{Estimates on the compressible part}

Applying the Fourier transform on the linearized system \eqref{3.3} gives that
\begin{eqnarray}\label{3.30}
\left\{ \begin{array}{llllll}
 \partial_t \hat{\rho} + |\xi| \hat{d} =\widehat{S^1},\\[2mm]
 \partial_t \hat{ d } - |\xi| \hat{\rho} + \nu |\xi|^2 \hat{d} -|\xi| \hat{\theta}  -\frac{1}{3\mathcal{C}} |\xi| \hat{j_0}=\widehat{D},\\[2mm]
\partial_t \hat{\theta} + \frac{2}{3} |\xi| \hat{d} + \frac23 \kappa |\xi|^2 \hat{\theta}  + \frac{2}{3} \gamma \hat{\theta}- \frac{2}{3} b \hat{j_0}=\widehat{S^3},\\[2mm]
 \partial_t \hat{j_0 } - \mathcal{C}\gamma \hat{\theta} + a |\xi|^2 \hat{j_0} + \mathcal{C} b \hat{j_0}=\widehat{S^4},
\end{array}
\right.
\end{eqnarray}
In other words, it is
\begin{eqnarray}\label{3.31}
 \frac{\rm d}{{\rm d}t}\left(\begin{array}{lll} \hat{\rho}\\ \hat{ d }\\ \hat{\theta}\\ \hat{j_0 }\end{array}\right)+ A(\xi)\left(\begin{array}{lll} \hat{\rho}\\ \hat{ d }\\ \hat{\theta}\\ \hat{j_0 }\end{array}\right)
 =\left(\begin{array}{lll} \widehat{S^1}\\ \widehat{D}\\ \widehat{S^3}\\ \widehat{S^4}\end{array}\right),
 \end{eqnarray}
with
\begin{eqnarray}\label{3.311}
 A(\xi)= \left(\begin{array}{ccccc} 0 & |\xi|& 0& 0 \\ -|\xi|& \nu|\xi|^2&-|\xi|&-\frac{1}{3\mathcal{C}}|\xi| \\0& \frac{2|\xi|}{3}& \frac{2}{3}\kappa|\xi|^2+\frac{2\gamma}{3}& -\frac{2}{3}b \\ 0& 0& -\mathcal{C}\gamma & a|\xi|^2 + \mathcal{C}b \end{array}\right).
\end{eqnarray}

Denote $ \varrho = |\xi|$, the characteristic polynomial of the matrix $A$ takes the following form.
\begin{equation}
\begin{split}
 P(\lambda)
 =& \left|A(\varrho)-\lambda I\right| \notag \\
 :=& a_0\lambda^4- a_1 \lambda^3 +a_2 \lambda^2 -a_3 \lambda +a_4, \label{eigenF1}
\end{split}
\end{equation}
where
\begin{equation*}
\begin{split}
 a_0=&1,\ \ \ \
 a_1
 =(a+\frac23 \kappa+\nu) \varrho^2 + \frac{2\gamma}{3}+\mathcal{C}b,\\
 a_2
 =& \left[(\nu a +\frac23 \nu \kappa+ \frac23 a \kappa)\varrho^2+\frac{2\gamma \nu}{3}+\mathcal{C} b\nu +\frac{2\gamma a}{3}+\frac{2 \kappa \mathcal{C} b}{3}+\frac{5}{3}\right]\varrho^2,\\
 a_3
 =& \left[\frac23 a \kappa \nu \varrho^4 +(\frac{2(\gamma a + \kappa \mathcal{C} b)\nu}{3} +\frac{5a}{3}+\frac{2\kappa}{ 3 })\varrho^2 +\frac{5}{3}\mathcal{C} b+\frac{8\gamma}{9}\right]\varrho^2,\\
 a_4
 =& \frac23 \left(a \kappa \varrho^2+a\gamma +\kappa \mathcal{C} b\right)\varrho^4.\\
\end{split}
\end{equation*}
It is direct to check that the matrix $A(0)$ has four eigenvalues, i.e. $0$ (three multiply) and $\frac{2\gamma}{3}+ \mathcal Cb$.
By a direct calculation, it is clear that for small positive $\varrho$, $P(\lambda)$ has a pair of complex conjugated eigenvalues and two real eigenvalues with asymptotic
expansion.
\begin{equation*}
\pm i
\rho
\sqrt{\frac{8\gamma+15\mathcal{C}b}{6\gamma+9\mathcal{C}b}}
+
\Big[
\frac{9\gamma}{(6\gamma+9\mathcal{C}b)^2}
+
\frac{2(a\gamma+\kappa\mathcal{C}b)(\gamma+3\mathcal{C}b)}{(8\gamma+15\mathcal{C}b)(2\gamma+3\mathcal{C}b)}
+
\frac{\nu}{2}
\Big]
\varrho^2
+
O(\varrho^3),
\end{equation*}
and
\begin{equation*}
\frac{6 ( a\gamma+\kappa \mathcal{C} b)}{15\mathcal{C} b + 8\gamma}\varrho^2+O(\varrho^3),
\ \ \ \
\frac{2\gamma}{3}+\mathcal{C} b
+
\frac{4\gamma \kappa+9\mathcal{C} b a-\gamma}{6\gamma+9\mathcal{C} b}
\varrho^2
+
O(\varrho^3).
\end{equation*}
According to these eigenvalues, we expect to have three dissipative modes (corresponding to $\lambda_i(\varrho),\ i=1,2,3$) and a damped mode (corresponding to $\lambda_4(\varrho)$). Roughly speaking, we hope that the modes corresponding to $\lambda_i(\varrho)\ (i=1,2,3)$ should have decay estimates of the type ${\rm e}^{-c_0 \varrho^2 t}$ with small enough $\varrho$ and a positive constant $c_0$ depended only on $\gamma,\ \kappa,\ \nu, \mathcal{C},\ a$ and $\ b$. The mode corresponding to $\lambda_4(\varrho)$ should have better decay properties.
In what follows, we use the modes:
\begin{equation}\label{3.33}
 \widehat{\Theta} :=3\mathcal{C} \widehat{\theta}+2\widehat{j_0},~~~~\widehat{\Xi}:=\gamma \widehat{\theta}-b\widehat{j_0}.
\end{equation}
Actually, to some extent, $\hat\rho$, $\hat d$ and $\widehat\Theta$ are ``dissipative modes'', and $\widehat\Xi$ is a ``damped mode''.
Thus the system \eqref{3.30} could be rewritten as
\begin{equation}\label{3.34}
\left\{
\begin{array}{llllll}
 \partial_t \widehat{\rho} + |\xi| \widehat{d}  =\widehat{\mathcal{S}^1}, \\[2mm]
 \partial_t \widehat{ d } - |\xi| \widehat{\rho} + \nu |\xi|^2 \widehat{d} - c_1 |\xi| \widehat{\Theta}  - c_2 |\xi| \widehat{\Xi}=\widehat{\mathcal{S}^2},\\[2mm]
\partial_t \widehat{\Theta} + 2\mathcal{C}|\xi| \widehat{d} + c_3 |\xi|^2 \widehat{\Theta}  -
c_4|\xi|^2\widehat{\Xi} =\widehat{\mathcal{S}^3},\\[2mm]
 \partial_t \widehat{\Xi} +\frac{2\gamma}{3}|\xi|\widehat{d} + c_5 \widehat{\Xi} -c_6 |\xi|^2 \widehat{\Theta} =\widehat{\mathcal{S}^4},
\end{array}
\right.
\end{equation}
where $(\mathcal{S}_1,\mathcal{S}_2,\mathcal{S}_3,\mathcal{S}_4)=(S_1,D,3\mathcal{C} S_3+2 S_4,\gamma S_3-b S_4)$ and
the coefficients $c_i$ $(i=1,\cdots,6)$ are defined by
\begin{equation*}
\begin{split}
&
c_1= \frac{b+\frac{\gamma}{3\mathcal{C}}}{3b\mathcal{C}+2\gamma} >0,~~~~c_2=\frac{1}{3b\mathcal{C}+2\gamma}>0,\\
&
c_3= \left(\frac{2\kappa b\mathcal{C}}{3b\mathcal{C}+2\gamma}+\frac{2a\gamma}{3b\mathcal{C}+2\gamma}\right)>0,\\
&
c_4= \left(\frac{6a\mathcal{C}}{3b\mathcal{C}+2\gamma}-\frac{4\kappa \mathcal{C}}{3b\mathcal{C}+2\gamma} \right),\\
&
c_5=\left[\frac{(4\kappa \gamma +9ab\mathcal{C})|\xi|^2}{3(3b\mathcal{C}+2\gamma)}+\Big(\frac{2\gamma}{3}+b \mathcal{C}\Big) \right]>0,\\
&
c_6= \frac{2\kappa \gamma b-3ab\gamma}{3(3b\mathcal{C}+2\gamma)}.
\end{split}
\end{equation*}
By the energy method in the Fourier space, we can get
\begin{equation} \label{3.35}
\begin{split}
 &
 \frac{1}{2}
 \frac{\rm d}{{\rm d}t}
 \left\{ |\widehat{\rho}|^2 + |\widehat{d}|^2 +\frac{c_1}{2\mathcal{C}} |\widehat{\Theta}|^2 +\frac{3c_2}{2\gamma} |\widehat{\Xi}|^2 \right\}
 +
 \nu |\xi|^2 |\widehat{d}|^2 +\frac{c_1c_3}{2\mathcal{C}} |\xi|^2|\widehat{\Theta}|^2 +\frac{3c_2c_5}{2\gamma} |\widehat{\Xi}|^2 \\
=
 &
  \left(\frac{c_1c_4}{2\mathcal{C}}+\frac{3c_2c_6}{2\gamma}\right)|\xi|^2 {\rm Re}(\widehat{\Theta}\overline{\widehat{\Xi}})\\
&
 +
 {\rm Re}\left(\widehat{\mathcal{S}^1}\overline{\widehat{\rho}}\right)
 +
 {\rm Re}\left(\widehat{\mathcal{S}^2}\overline{\widehat{d}}\right)
 +
 \frac{c_1}{2\mathcal{C}}
 {\rm Re}\left(\widehat{\mathcal{S}^3}\overline{\widehat{\Theta}}\right)
 +
 \frac{3c_2}{2\gamma}
 {\rm Re}\left(\widehat{\mathcal{S}^4}\overline{\widehat{\Xi}}\right).
\end{split}
\end{equation}
By the similar arguments, it follows from $\eqref{3.34}_1$ and $\eqref{3.34}_2$ that
\begin{equation}\label{3.36}
\begin{split}
\frac{\rm d}{{\rm d}t}
{\rm Re}(\widehat{\rho}\overline{\widehat{d}})
-
|\xi| |\widehat{\rho}|^2
+
|\xi| |\widehat{d}|^2
=
&
 -
 \nu |\xi|^2 {\rm Re}(\widehat{\rho}\overline{\widehat{d}})
 +
 c_1|\xi| {\rm Re}(\widehat{\rho}\overline{\widehat{\Theta}})
 +
 c_2|\xi| {\rm Re}(\widehat{\rho}\overline{\widehat{\Xi}})\\
 &
 +
 {\rm Re}\left(\widehat{\mathcal{S}^1}\overline{\widehat{d}}\right)
 +
 {\rm Re}\left(\widehat{\mathcal{S}^2}\overline{\widehat{\rho}}\right).
\end{split}
\end{equation}
And $\eqref{3.35}-\beta_3|\xi| \times \eqref{3.36}$ gives
\begin{equation} \label{3.37}
\begin{split}
 &
 \frac{1}{2}
 \frac{\rm d}{{\rm d}t}
 \left\{
 |\widehat{\rho}|^2
 -
 2\beta_3 |\xi| {\rm Re}(\widehat{\rho} \overline{\widehat{d}})
 +
 |\widehat{d}|^2
 +
 \frac{c_1}{2\mathcal{C}}
  |\widehat{\Theta}|^2
 +
 \frac{3c_2}{2\gamma}
 |\widehat{\Xi}|^2
 \right\}\\
 &
 +\beta_3 |\xi|^2|\widehat{\rho}|^2
 +
  (\nu-\beta_3) |\xi|^2 |\widehat{d}|^2
  +
  \frac{c_1c_3}{2\mathcal{C}} |\xi|^2|\widehat{\Theta}|^2
  +
  \frac{3c_2c_5}{2\gamma} |\widehat{\Xi}|^2\\
=
 &
 \left(\frac{c_1c_4}{2\mathcal{C}}+\frac{3c_2c_6}{2\gamma}\right)
 |\xi|^2
 {\rm Re}(\widehat{\Theta}\overline{\widehat{\Xi}})
 +
 \beta_3 \nu
 |\xi|^3
 {\rm Re}(\widehat{\rho}\overline{\widehat{d}})
 -
 \beta_3 c_1
 |\xi|^2{\rm Re}(\widehat{\rho}\overline{\widehat{\Theta}})\\
 &
 -
 \beta_3 c_2
 |\xi|^2{\rm Re}(\widehat{\rho}\overline{\widehat{\Xi}})
 +
 {\rm Re}\left(\widehat{\mathcal{S}^1}\overline{\widehat{\rho}}\right)
 +
 {\rm Re}\left(\widehat{\mathcal{S}^2}\overline{\widehat{d}}\right)
 +
 \frac{c_1}{2\mathcal{C}}
 {\rm Re}\left(\widehat{\mathcal{S}^3}\overline{\widehat{\Theta}}\right)\\
 &
 +
 \frac{3c_2}{2\gamma}
 {\rm Re}\left(\widehat{\mathcal{S}^4}\overline{\widehat{\Xi}}\right)
 -
 \beta_3
 {\rm Re}\left(\widehat{\mathcal{S}^1}\overline{\widehat{d}}\right)
 -
 \beta_3
 {\rm Re}\left(\widehat{\mathcal{S}^2}\overline{\widehat{\rho}}\right),
\end{split}
\end{equation}
for some positive constant $\beta_3$.
By the Young inequality, a direct calculation yields
\begin{equation} 
\begin{split}
  {\rm R.H.S.\ terms\ of}\ \eqref{3.37}
 \leq
 &
  \frac{c_1c_3}{4\mathcal{C}}
  |\xi|^2 |\widehat{\Theta}|^2
  +
 \left(\frac{c_1|c_4|}{4\mathcal{C}}+\frac{3c_2|c_6|}{4\gamma}\right)^2
  \frac{\mathcal{C}}{c_1c_3} |\xi|^2 |\widehat{\Xi}|^2\\
 &
 +
 \frac{\beta_3 \nu}{2} |\xi|^3(|\widehat{\rho}|^2+|\widehat{d}|^2)
 +
  \frac{\beta_3}{4} |\xi|^2|\widehat{\rho}|^2+2\beta_3 c_1^2|\xi|^2|\widehat{\Theta}|^2\\
 &
 +
 2\beta_3 c_2^2|\xi|^2|\widehat{\Xi}|^2.
\end{split}
\end{equation}
Then we have
\begin{equation}\label{3.39}
\begin{split}
 &
 \frac{\rm d}{{\rm d}t}
 \mathcal{L}_l(t,\xi)
 +\frac{3\beta_3}{4} |\xi|^2|\widehat{\rho}|^2
 +
  (\nu-\beta_3) |\xi|^2 |\widehat{d}|^2
  +
  \left(\frac{c_1c_3}{4\mathcal{C}}-2\beta_3 c_1^2\right) |\xi|^2|\widehat{\Theta}|^2
  +
  \frac{3c_2c_5}{2\gamma} |\widehat{\Xi}|^2\\
\leq
 &
 \frac{\beta_3 \nu}{2} |\xi|^3(|\widehat{\rho}|^2+|\widehat{d}|^2)
 +
 \left[
 \left(\frac{c_1|c_4|}{4\mathcal{C}}+\frac{3c_2|c_6|}{4\gamma}\right)^2
  \frac{\mathcal{C}}{c_1c_3}
  +
  2\beta_3 c_2^2
 \right]
   |\xi|^2 |\widehat{\Xi}|^2
 +
 {\rm Re}\left(\widehat{\mathcal{S}^1}\overline{\widehat{\rho}}\right)\\
 &
 +
 {\rm Re}\left(\widehat{\mathcal{S}^2}\overline{\widehat{d}}\right)
 +
 \frac{c_1}{2\mathcal{C}}
 {\rm Re}\left(\widehat{\mathcal{S}^3}\overline{\widehat{\Theta}}\right)
 +
 \frac{3c_2}{2\gamma}
 {\rm Re}\left(\widehat{\mathcal{S}^4}\overline{\widehat{\Xi}}\right)
 -
 \beta_3
 {\rm Re}\left(\widehat{\mathcal{S}^1}\overline{\widehat{d}}\right)
 -
 \beta_3
 {\rm Re}\left(\widehat{\mathcal{S}^2}\overline{\widehat{\rho}}\right),
\end{split}
\end{equation}
where the functional $\mathcal{L}_l(t,\xi)$ takes the following form.
\begin{equation*}
\mathcal{L}_l(t,\xi)
 :=
  \frac{1}{2}
   |\widehat{\rho}|^2
   -
   \beta_3 |\xi| {\rm Re}(\widehat{\rho} \overline{\widehat{d}})
   +
   \frac12 |\widehat{d}|^2
   +
   \frac{c_1}{4\mathcal{C}} |\widehat{\Theta}|^2
   +
   \frac{3c_2}{4\gamma} |\widehat{\Xi}|^2.
\end{equation*}
We denote that $r_0$ and $R_0$ are two fixed positive constants satisfying
\begin{equation}
 r_0:=
 \min\Big\{
  \frac{1}{2\nu},
\left[\left(\frac{c_1|c_4|}{4\mathcal{C}}+\frac{3c_2|c_6|}{4\gamma}\right)^2\frac{\mathcal{C}}{c_1c_3}\right]^{-\frac{1}{2}}
  \frac{3c_2c_5}{4\gamma},
  \frac{1}{2}
  \Big\},
\end{equation}
and
\begin{equation}
R_0:=2^{k_1+1},
\end{equation}
where the positive integer $k_1$ is defined by \eqref{3.13}.
Moreover, the constant $\beta_3$ is chosen to be a small constant which satisfies
\begin{equation}
0<
\beta_3
\leq
\min\Big\{
\frac{2\nu}{3},
\frac{c_3}{16\mathcal{C}c_1},
\frac{1}{2R_0}
\Big\}.
\end{equation}
Since $\beta_3$ is small, then we get, for any $|\xi|\leq R_0$, that
\begin{equation}\label{3.43}
 \mathcal{L}_l(t,\xi)
 \sim
  |\widehat{\rho}|^2+ |\widehat{d}|^2 + |\widehat{\Theta}|^2+ |\widehat{\Xi}|^2,
\end{equation}
which implies that there exists a positive constant $C_4=C_4(\beta_3,\nu,c_1,c_2,c_3,c_5,R_0)$, such that
\begin{equation}\label{3.44}
 C_4
 |\xi|^2\mathcal{L}_l(t,\xi)
 \leq
  \frac{3\beta_3}{4} |\xi|^2|\widehat{\rho}|^2
 +
  (\nu-\beta_3) |\xi|^2 |\widehat{d}|^2
  +
  \left(\frac{c_1c_3}{4\mathcal{C}}-2\beta_3 c_1^2\right) |\xi|^2|\widehat{\Theta}|^2
  +
  \frac{3c_2c_5}{2\gamma} |\widehat{\Xi}|^2.
\end{equation}
From \eqref{3.39} and \eqref{3.44}, we have, for $0\leq |\xi|\leq R_0$,
\begin{equation}\label{3.45}
\begin{split}
&
 \frac{\rm d}{{\rm d}t}
  \mathcal{L}_l(t,\xi)
 +
 |\xi|^2\mathcal{L}_l(t,\xi)\\
\leq
 &
 \frac{\beta_3 \nu}{2} |\xi|^3(|\widehat{\rho}|^2+|\widehat{d}|^2)
 +
 \left[
 \left(\frac{c_1|c_4|}{4\mathcal{C}}+\frac{3c_2|c_6|}{4\gamma}\right)^2
  \frac{\mathcal{C}}{c_1c_3}
  +
  2\beta_3 c_2^2
 \right]
   |\xi|^2 |\widehat{\Xi}|^2
 +
 {\rm Re}\left(\widehat{\mathcal{S}^1}\overline{\widehat{\rho}}\right)\\
 &
 +
  {\rm Re}\left(\widehat{\mathcal{S}^2}\overline{\widehat{d}}\right)
 +
 \frac{c_1}{2\mathcal{C}}
 {\rm Re}\left(\widehat{\mathcal{S}^3}\overline{\widehat{\Theta}}\right)
 +
 \frac{3c_2}{2\gamma}
 {\rm Re}\left(\widehat{\mathcal{S}^4}\overline{\widehat{\Xi}}\right)
 -
 \beta_3
 {\rm Re}\left(\widehat{\mathcal{S}^1}\overline{\widehat{d}}\right)
 -
 \beta_3
 {\rm Re}\left(\widehat{\mathcal{S}^2}\overline{\widehat{\rho}}\right).
\end{split}
\end{equation}
Specially, for $0\leq |\xi|\leq r_0$, from \eqref{3.39} and \eqref{3.44}, we have
\begin{equation}\label{3.46}
\begin{split}
&
 \frac{\rm d}{{\rm d}t}
  \mathcal{L}_l(t,\xi)
 +
 |\xi|^2\mathcal{L}_l(t,\xi)\\
\lesssim
 &
 {\rm Re}\left(\widehat{\mathcal{S}^1}\overline{\widehat{\rho}}\right)
 +
  {\rm Re}\left(\widehat{\mathcal{S}^2}\overline{\widehat{d}}\right)
 +
 \frac{c_1}{2\mathcal{C}}
 {\rm Re}\left(\widehat{\mathcal{S}^3}\overline{\widehat{\Theta}}\right)
 +
 \frac{3c_2}{2\gamma}
 {\rm Re}\left(\widehat{\mathcal{S}^4}\overline{\widehat{\Xi}}\right)\\
 &
 -
 \beta_3
 {\rm Re}\left(\widehat{\mathcal{S}^1}\overline{\widehat{d}}\right)
 -
 \beta_3
 {\rm Re}\left(\widehat{\mathcal{S}^2}\overline{\widehat{\rho}}\right).
\end{split}
\end{equation}
From \eqref{3.33}, it is obvious that
\begin{equation}
 \widehat{\theta}= \frac{b\widehat{\Theta}+ 2\widehat{\Xi}}{3\mathcal{C}b+2\gamma} ,~~~~\widehat{j_0}= \frac{\gamma \widehat{\Theta}-3\mathcal{C} \widehat{\Xi}}{3\mathcal{C} b+2\gamma}.
\end{equation}
Then there exists some positive constant $C_5$, such that
\begin{equation}\label{3.48}
 C_5^{-1} \left(|\widehat{\theta}|^2 +|\widehat{j_0}|^2 \right) \leq |\widehat{\Theta}|^2 +|\widehat{\Xi}|^2  \leq C_5 \left(|\widehat{\theta}|^2 +|\widehat{j_0}|^2 \right),
\end{equation}
which and \eqref{3.43} imply that
\begin{equation}\label{3.49}
\mathcal{L}_l(t,\xi)
 \sim
  |\widehat{\rho}|^2+ |\widehat{d}|^2 + |\widehat{\theta}|^2+ |\widehat{j_0}|^2.
\end{equation}

Using the inequalities that have been obtained in the different frequency regimes and the Parseval formula, we deduce the estimates for $(\rho_k,d_k,\theta_k,j_{0,k})$. More precisely, from \eqref{3.46} and \eqref{3.49}, by taking $k_0=[\log_2r_0]-1$, we get for any $k\leq k_0$
\begin{equation}\label{3.50}
\begin{split}
&
\|(\rho_k , d_k,\theta_k,j_{0,k})(t) \|_{L^2}^2
+
\int_0^t
\Big(
2^{2k}
\|(\rho_k,d_k ,\theta_k,j_{0,k})(\tau)\|_{L^2}^2
+
\|(\gamma\theta_k-b j_{0,k})(\tau)\|_{L^2}^2
\Big)
{\rm d}\tau\\
\lesssim
&
\|(\rho_k,d_k ,\theta_k,j_{0,k})(0)\|_{L^2}^2
+
\int_0^t
\int_{\mathbb{R}^3}
\hat\varphi_j^2(\xi)
{\rm Re}\left(\widehat{\mathcal{S}^1}\overline{\widehat{\rho}}\right)
 {\rm d}\xi
 {\rm d}\tau\\
 &
 +
 \int_0^t
\int_{\mathbb{R}^3}
\hat\varphi_j^2(\xi)
\left\{
  {\rm Re}\left(\widehat{\mathcal{S}^2}\overline{\widehat{d}}\right)
  +
 \frac{c_1}{2\mathcal{C}}
 {\rm Re}\left(\widehat{\mathcal{S}^3}\overline{\widehat{\Theta}}\right)
 +
 \frac{3c_2}{2\gamma}
 {\rm Re}\left(\widehat{\mathcal{S}^4}\overline{\widehat{\Xi}}\right)
 \right\}
 {\rm d}\xi
 {\rm d}\tau\\
 &
 +
 \int_0^t
\int_{\mathbb{R}^3}
\hat\varphi_j^2(\xi)
\left\{
 -
 \beta_3
 {\rm Re}\left(\widehat{\mathcal{S}^1}\overline{\widehat{d}}\right)
 -
 \beta_3
 {\rm Re}\left(\widehat{\mathcal{S}^2}\overline{\widehat{\rho}}\right)
 \right\}
 {\rm d}\xi
 {\rm d}\tau.
\end{split}
\end{equation}
By using the H${\rm \ddot{o}}$lder inequality and the Young inequality, the second term on the right hand side of \eqref{3.50} can be estimated as follows.
\begin{equation*}
\begin{split}
\int_0^t
\int_{\mathbb{R}^3}
\hat\varphi_j^2(\xi)
{\rm Re}\left(\widehat{\mathcal{S}^1}\overline{\widehat{\rho}}\right)
{\rm d}\xi
{\rm d}\tau
\leq
&
\int_0^t
\|\hat\varphi_j(\xi)\widehat{\mathcal{S}^1}\|_{L^2}
\|\hat\varphi_j(\xi)\overline{\widehat{\rho}}\|_{L^2}
{\rm d}\tau\\
\leq
&
\epsilon
\int_0^t
2^{2k}
\|\rho_k\|_{L^2}^2
{\rm d}\tau
+
C_\epsilon
\int_0^t
2^{-2k}\|\mathcal{S}^1_k\|_{L^2}^2
{\rm d}\tau,
\end{split}
\end{equation*}
where $\epsilon$ is a small positive constant.
The other terms on the right hand side of \eqref{3.50} can be estimated in the same way. Therefore, for $k\leq k_0$, we have
\begin{equation}\label{3.51}
\begin{split}
&
\|(\rho_k , d_k,\theta_k,j_{0,k})(t) \|_{L^2}^2
+
\int_0^t
\Big(
2^{2k}
\|(\rho_k,d_k ,\theta_k,j_{0,k})(\tau)\|_{L^2}^2
+
\|(\gamma\theta_k-b j_{0,k})(\tau)\|_{L^2}^2
\Big)
{\rm d}\tau\\
\lesssim
&
\|(\rho_k,d_k ,\theta_k,j_{0,k})(0)\|_{L^2}^2
+
C_\epsilon
\int_0^t
2^{-2k}
\|(\mathcal{S}^1_k,\mathcal{S}^2_k,\mathcal{S}^3_k,\mathcal{S}^4_k)(\tau)\|_{L^2}^2
{\rm d}\tau.
\end{split}
\end{equation}

\subsubsection{Estimates on the incompressible part $(\mathcal{P} u)_k$}

Below, we will derive the estimates of $\widehat{(\mathcal{P} u)_k}$. The linearized equations of \eqref{3.6} in Fourier variables take the following form.
\begin{equation}\label{3.52}
 \partial_t\widehat{(\mathcal{P} u)_k}  + \mu|\xi|^2\widehat{(\mathcal{P} u)_k}  = \widehat{(\mathcal{P}S^2)_k}.
\end{equation}
By a direct calculation, it follows from \eqref{3.52} and the Bernstein inequality that for all $|\xi|\geq 0$
\begin{equation}\label{3.53}
\begin{split}
&
\|(\mathcal{P} u)_k(t)\|_{L^2}^2
+
2\mu
\int_0^t
2^{2k}
\|(\mathcal{P} u)_k(\tau)\|_{L^2}^2{\rm d}\tau\\
\lesssim
&
 \|(\mathcal{P} u)_k(0)\|_{L^2}^2
 +
 \int_0^t
 \int_{\mathbb{R}^3}
 \hat\varphi_j(\xi)^2
 {\rm Re}
 (
  \widehat{(\mathcal{P} S^2)}
  \overline{\widehat{(\mathcal{P} u)}}
 )
 {\rm d}\xi
 {\rm d}\tau.
\end{split}
\end{equation}
Similar as the estimate \eqref{3.51}, for any integer $k$, we have
\begin{equation}\label{3.54}
\begin{split}
\|(\mathcal{P} u)_k(t)\|_{L^2}^2
+
\mu
\int_0^t
2^{2k}
\|(\mathcal{P} u)_k(\tau)\|_{L^2}^2{\rm d}\tau
\lesssim
 \|(\mathcal{P} u)_k(0)\|_{L^2}^2
 +
 C_\epsilon
 \int_0^t
 2^{-2k}
 \|(\mathcal{P} S^2)_k\|_{L^2}^2
 {\rm d}\tau.
\end{split}
\end{equation}

\subsection{Medium-frequency analysis}

According to the analysis in Section 3.1.2, it could be checked that the eigenvalues of $A(\xi)$ have positive real parts for small enough $\varrho$. In order to show that no condition is required for large $\varrho$, by Routh-Hurwitz theorem, the roots of the function $P(\lambda)$ have positive real part if and only if the following determinants are positive:
\begin{equation*}
A_1:= a_1,\ \
A_2
:=
\left|
  \begin{array}{cc}
    a_1 & a_0 \\
    a_3 & a_2 \\
  \end{array}
\right|,\ \
A_3
:=
\left|
  \begin{array}{ccc}
    a_1 & a_0 & 0 \\
    a_3 & a_2 & a_1 \\
    0 & a_4 & a_3 \\
  \end{array}
\right|,\ \
A_4
:=
\left|
  \begin{array}{cccc}
    a_1 & a_0 & 0 & 0 \\
    a_3 & a_2 & a_1 & 0 \\
    0 & a_4 & a_3 & 0 \\
    0 & 0 & 0 & a_4 \\
  \end{array}
\right|.
\end{equation*}
It is clear that $A_1>0$ and ${\rm sgn}A_3={\rm sgn}A_4$. It is easy to check that
\begin{equation*}
\begin{split}
A_2
 =
 &
  a_1a_2-a_0a_3\\
 :=
 &
 a_{21}\varrho^6
 +
 a_{22}\varrho^4
 +
 a_{23}\varrho^2
 >
  0,
\end{split}
\end{equation*}
where the coefficients $a_{21},\ a_{22}$ and $a_{23}$ are defined by
\begin{equation*}
\begin{split}
 a_{21}
 :=
 &
  (a+\frac{2}{3}\kappa+\nu)
  (\nu a +\frac{2}{3}\nu\kappa +\frac{2}{3} a \kappa)
  -
  \frac{2}{3} a\kappa \nu
  >0,\\
a_{22}
 :=
 &
  (\frac{2\gamma}{3}+\mathcal{C}b)
  (\nu a +\frac{2}{3}\nu\kappa +\frac{2}{3} a\kappa)
  +
  (a+\frac{2\kappa}{3}+\nu)
  (\frac{2\gamma\nu}{3}+\mathcal{C}b\nu+\frac{2\gamma a}{3}+\frac{2\kappa\mathcal{C}b}{3}+\frac{5}{3})\\
  &
  -
  (\frac{2(\gamma a+\kappa\mathcal{C}b)\nu}{3}+\frac{5a}{3}+\frac{2\kappa}{3})
  >0,
  \\
 a_{23}
 :=
 &
  (\frac{2\gamma}{3}+\mathcal{C}b)
  (\frac{2\gamma\nu}{3}+\mathcal{C}b\nu+\frac{2\gamma a}{3}+\frac{2\kappa\mathcal{C}b}{3}+\frac{5}{3})
  -
  (\frac{5}{3}\mathcal{C}b+\frac{8\gamma}{9})
  >0.
  \\
\end{split}
\end{equation*}
By a direct but tedious calculation, we get
\begin{equation*}
\begin{split}
A_3
 =
 &
  a_3(a_1a_2-a_0a_3)
  -
  a_1^2a_4
>
  0.
\end{split}
\end{equation*}

Following the arguments in Section 3.3 of \cite{Danchin-Ducomet}, we have the following lemma.
\begin{lemma}[\cite{Danchin-Ducomet}]\label{Lemma 3.1}
For any given constants $r$ and $R$ with $0<r<R$, there exists a positive constant $\iota$ (depending only on $r,R$ and the constants $\nu$, $k$, $\gamma$, $\kappa$, $\mathcal C$,  $a$ and $b$) such that
\begin{eqnarray}\label{3.55}
 |{\rm e}^{-t A(\xi)}| \leq C {\rm e}^{-\iota t} ~~{\rm for\ all}~~r\leq |\xi|\leq R~~ {\rm and}~~t\in \mathbb{R}^+.
\end{eqnarray}
\end{lemma}
For the system \eqref{3.30} and the Duhamel principle, the inequality \eqref{3.55} yields for all $r\leq |\xi|\leq R$
\begin{equation}\label{3.56}
\begin{split}
|(\widehat{\rho} , \widehat{d} ,\widehat{\theta},\widehat{j_{0}})(t,\xi) |
\lesssim
&
{\rm e}^{-\iota t}
|(\widehat{\rho} , \widehat d ,\widehat{\theta},\widehat{j_0})(0,\xi)|
+
\int_0^t
{\rm e}^{-\iota(t-\tau)}
|(\widehat{S^1},\widehat{D},\widehat{S^3},\widehat{S^4})(\tau,\xi)|
{\rm d}\tau,
\end{split}
\end{equation}
where $r$ and $R$ are any given positive constants.
\begin{proposition}\label{Proposition 4.2}
For any integer $k$ with $k_0\leq k\leq k_1$, there exists a positive constant $C_7$ depending on $k_1$, such that
\begin{equation}\label{3.57}
\begin{split}
\int_0^t
\|(\rho_k,d_k,\theta_k,j_{0,k})(\tau)\|_{L^2}^2
{\rm d}\tau
\leq
&
C\|(\rho_k,d_k,\theta_k,j_{0,k})(0)\|_{L^2}^2\\
&
+
C\int_0^t
\|(S^1_k,S^2_k,S^3_k,S^4_k)(\tau)\|_{L^2}^2
{\rm d}\tau.
\end{split}
\end{equation}
\end{proposition}

\begin{proof}
Taking $r=2^{k_0-1}$ and $R=R_0$ and using \eqref{3.57}, we have for $k_0\leq k\leq k_1$
\begin{equation}
\begin{split}
\|(\rho_k,d_k,\theta_k,j_{0,k}) (t)\|_{L^2}
\leq
&
C{\rm e}^{-Ct}\|(\rho_k,d_k,\theta_k,j_{0,k})(0)\|_{L^2}\\
&
+
C\int_0^t
{\rm e}^{-C(t-\tau)}
\|(S^1_k,D_k,S^3_k,S^4_k)(\tau)\|_{L^2}
{\rm d}\tau.
\end{split}
\end{equation}
By a direct calculation, we have
\begin{equation}\label{3.59}
\begin{split}
\int_0^t
\|(\rho_k,d_k,\theta_k,j_{0,k}) (s)\|_{L^2}^2{\rm d}s
\leq
&
C\|(\rho_k,d_k,\theta_k,j_{0,k})(0)\|_{L^2}^2\\
&
+
C\int_0^t
{\rm d}s
\left(
\int_0^s
{\rm e}^{-C(s-\tau)}
\|(S^1_k,D_k,S^3_k,S^4_k)(\tau)\|_{L^2}
\right)^2.
\end{split}
\end{equation}
By using the H${\rm \ddot{o}}$ler inequality and exchanging the order of integration, the last term of \eqref{3.59} can be estimated as follows.
\begin{equation}\label{3.60}
\begin{split}
&
\int_0^t
{\rm d}s
\left(
\int_0^s
{\rm e}^{-C(s-\tau)}\|(S^1_k,D_k,S^3_k,S^4_k)(\tau)\|_{L^2}
{\rm d}\tau
\right)^2\\
\leq
&
\int_0^t
{\rm d}s
\left(
\int_0^s
{\rm e}^{-C(s-\tau)}
{\rm d}\tau
\right)
\left(
\int_0^s
{\rm e}^{-C(s-\tau)}\|(S^1_k,D_k,S^3_k,S^4_k)(\tau)\|_{L^2}^2
{\rm d}\tau
\right) \\
\leq
&
C
\int_0^t
{\rm d}s
\int_0^s
{\rm e}^{-C(s-\tau)}\|(S^1_k,D_k,S^3_k,S^4_k)(\tau)\|_{L^2}^2
{\rm d}\tau\\
\leq
&
C
\int_0^t
{\rm d}\tau
\int_\tau^t
{\rm e}^{-C(s-\tau)}
\|(S^1_k,D_k,S^3_k,S^4_k)(\tau)\|_{L^2}^2
{\rm d}s\\
\leq
&
C\int_0^t
\|(S^1_k,D_k,S^3_k,S^4_k)(\tau)\|_{L^2}^2
{\rm d}\tau.
\end{split}
\end{equation}
Plugging \eqref{3.60} into \eqref{3.59} yields \eqref{3.57}.
\end{proof}

Similar as the estimate \eqref{3.51}, taking $R=R_0$, we have from \eqref{3.45} that for $k\leq k_1$
\begin{equation}\label{3.61}
\begin{split}
&
\|(\rho_k , d_k,\theta_k,j_{0,k})(t) \|_{L^2}^2
+
\int_0^t
\Big(
2^{2k}
\|(\rho_k,d_k ,\theta_k,j_{0,k})(\tau)\|_{L^2}^2
+
\|(\gamma\theta_k-b j_{0,k})(\tau)\|_{L^2}^2
\Big)
{\rm d}\tau\\
\lesssim
&
\|(\rho_k,d_k ,\theta_k,j_{0,k})(0)\|_{L^2}^2
+
\int_0^t
2^{3k}
\|(\rho_k,d_k)(\tau)\|_{L^2}^2
{\rm d}\tau
+
\int_0^t
2^{2k}
\|(\theta_k,j_{0,k})(\tau)\|_{L^2}^2
{\rm d}\tau\\
&
+
C_\epsilon
\int_0^t
2^{-2k}
\|(\mathcal{S}^1_k,\mathcal{S}^2_k,\mathcal{S}^3_k,\mathcal{S}^4_k)(\tau)\|_{L^2}^2
{\rm d}\tau.
\end{split}
\end{equation}
For $k_0\leq k\leq k_1$, it follows from \eqref{3.57} that
\begin{equation}\label{3.62}
\begin{split}
&
\int_0^t
2^{3k}
\|(\rho_k,d_k)(\tau)\|_{L^2}^2
{\rm d}\tau
+
\int_0^t
2^{2k}
\|(\theta_k,j_{0,k})(\tau)\|_{L^2}^2
{\rm d}\tau\\
\leq
&
C
(2^{3k_1}+2^{2k_1})
\|(\rho_k,d_k,\theta_k,j_{0,k})(0)\|_{L^2}^2\\
&
+
C
(2^{4k_1}+2^{3k_1})
\int_0^t
2^{-2k}
\|(S^1_k,S^2_k,S^3_k,S^4_k)(\tau)\|_{L^2}^2
{\rm d}\tau.
\end{split}
\end{equation}
Putting \eqref{3.62} into \eqref{3.61} and recalling the definitions of $\mathcal{S}_k^i$ ($i=1,2,3,4$) yield
\begin{equation}\label{3.63}
\begin{split}
&
\|(\rho_k , d_k,\theta_k,j_{0,k})(t) \|_{L^2}^2
+
\int_0^t
\Big(
2^{2k}
\|(\rho_k,d_k ,\theta_k,j_{0,k})(\tau)\|_{L^2}^2
+
\|(\gamma\theta_k-b j_{0,k})(\tau)\|_{L^2}^2
\Big)
{\rm d}\tau\\
\lesssim
&
\|(\rho_k,d_k ,\theta_k,j_{0,k})(0)\|_{L^2}^2
+
\int_0^t
2^{-2k}
\|(S^1_k,S^2_k,S^3_k,S^4_k)(\tau)\|_{L^2}^2
{\rm d}\tau\\
&
+
C_\epsilon
\int_0^t
2^{-2k}
\|(S^1_k,S^2_k,S^3_k,S^4_k)(\tau)\|_{L^2}^2
{\rm d}\tau,
\end{split}
\end{equation}
for $k_0\leq k\leq k_1$.

Now, we show the estimates of long wave parts for the solution to \eqref{3.1} as follows. According to the definition in \eqref{5.1} and the definitions of $\mathcal{S}^i\ (i=1,2,3,4)$, $D_k$ and $\mathcal{P}S^2_k$, multiplying by $2^{ks^\prime}$ and summing up over $k$ with $k\leq k_0$ for \eqref{3.51}, $k_0\leq k\leq k_1$ for \eqref{3.63} and $k\leq k_1$ for \eqref{3.54}, we get the following lemma.
\begin{proposition}\label{Proposition 3.4}
Let $s^\prime$ be any real number. The following inequality holds true for the solution to \eqref{3.1}
\begin{equation}\label{3.64}
\begin{split}
&
\| (\rho,  u ,\theta,j_{0})^L(t) \|^2_{\dot B_{2,2}^{s^\prime}}
 +
 \int_0^t
 \Big(
 \|(\rho,  u ,\theta,j_{0})^L(\tau)\|^2_{\dot B_{2,2}^{s^\prime+1}}
 +
 \|(\gamma \theta-b j_0)^L(\tau)\|^2_{\dot B_{2,2}^{s^\prime}}
 \Big)
 {\rm d}\tau\\
\lesssim
&
   \| (\rho,  u ,\theta,j_{0})^L(0) \|^2_{\dot B_{2,2}^{s^\prime}}
   +
   \int_0^t
   \|(S^1,S^2,S^3,S^4)^L(\tau)\|_{\dot B_{2,2}^{s^\prime-1}}^2
    {\rm d}\tau.
\end{split}
\end{equation}
\end{proposition}

\subsection{Estimates of the nonlinear problem}
It is ready for us to prove the global existence and uniqueness of solution stated in Theorem \ref{Theorem 1.1}.
Under the  uniform {\it a  priori} assumption \eqref{2.3}, by using the Sobolev imbedding inequality, we have
\begin{equation*}
\frac{1}{2}
 \leq
  \rho+1
   \leq
    \frac{3}{2}.
\end{equation*}
This will be used frequently in this section. Therefore, for some positive constant $C$, we obtain
\begin{equation}\label{3.65}
|g(\rho)|\leq C\rho,
\ \
|h(\rho)|\leq C
\ \
{\rm and}
\ \
|g^{(k)}(\rho)|,
\
|h^{(k)}(\rho)|
\leq
 C,
\ \
{\rm for}
\
k\geq 1.
\end{equation}
Firstly, for $t\in [0, T]$, we define
\begin{equation}\label{3.66}
N(t):=
\sup\limits_{0\leq \tau\leq t}
\|(\rho,u,\theta,j_0)(\tau)\|_{H^4}^2
+
\int_0^t
\Big(
\|\nabla \rho(\tau)\|_{H^3}^2
+
\|\nabla(u,\theta,j_0)(\tau)\|_{H^4}^2
+
\|(\gamma \theta-b j_0)(\tau)\|_{H^4}^2
\Big)
{\rm d}\tau.
\end{equation}
From Propositions \ref{Proposition 3.3} and \ref{Proposition 3.4} (with $s=s^\prime=4$ or $s=1,s^\prime=0$), Lemma \ref{Lemma 5.1} and Lemma \ref{Lemma 5.7} (with $q=2$, $p=\frac{6}{5}$ and $l=1$), by invoking the equivalence of the norms for $s\geq 0$
\begin{equation*}
\|f\|_{H^s} \sim \|f\|_{L^2}+\|f\|_{\dot H^s},
\end{equation*}
 we have
\begin{equation}\label{3.67}
\begin{split}
 N(t)
 \lesssim
 &
 N(0)
 +
 \int_0^t
 \Big(
 \| S^{12}(\tau)\|_{H^4}^2
 +
 \|(S^1,S^2,S^3,S^4,(\nabla u)^T \cdot\nabla\rho)(\tau)\|_{H^3}^2
 \Big)
 {\rm d}\tau \\
 &
 +
 \int_0^t
  \Big(
  \|\nabla^2\rho\|_{L^\infty}^2
  \|u^S\|_{H^3}^2
  +
  \|\nabla u (\tau)\|_{L^{\infty}}\|\nabla\rho^S(\tau)\|_{H^3}^2
 \Big) {\rm d}\tau\\
 &
 +
 \int_0^t
  \|\nabla u\|_{L^\infty}^2
  \sum\limits_{k>k_1}
  (1+2^{6k})
  \left(
 \sum\limits_{l\geq k-1}
 2^{k-l}
 \|\nabla \rho_l\|_{L^2}\right)^2
  {\rm d}\tau\\
 &
 +
 \int_0^t
 \|(S^1,S^2,S^3,S^4)(\tau)\|_{L^\frac{6}{5}}^2
 {\rm d}\tau,
\end{split}
\end{equation}
where we used the facts that
\begin{equation*}
\|(S^{12})^S\|_{\dot B^4_{2,2}}\lesssim \|S^{12}\|_{\dot H^4},
\end{equation*}
\begin{eqnarray*}
\|(S^1,S^2,S^3,S^4)^L\|_{\dot B^{4}_{2,2}} \lesssim \|(S^1,S^2,S^3,S^4)^L\|_{\dot B^{3}_{2,2}}\lesssim \|(S^1,S^2,S^3,S^4)\|_{\dot H^3},
\end{eqnarray*}
and
\begin{eqnarray*}
\begin{split}
\|(S^1,S^2,S^3,S^4)^L\|_{\dot B^{-1}_{2,2}}
\lesssim
&
\|\Lambda^{-1}(S^1,S^2,S^3,S^4)^L\|_{\dot B^{0}_{2,2}}\\
\lesssim
&
 \|\Lambda^{-1}(S^1,S^2,S^3,S^4)\|_{L^2}
\lesssim \|(S^1,S^2,S^3,S^4)\|_{L^\frac{6}{5}}.
\end{split}
\end{eqnarray*}
For the nonlinear terms on the r.h.s. of \eqref{3.67}, by using the Sobolev imbedding inequality in Lemmas \ref{Lemma 5.4} and \ref{Lemma 5.5}, we have
\begin{equation*}
\begin{split}
 \|S^{12}\|_{H^4}
 \lesssim
 &
 \|\rho\|_{L^{\infty}}\|{\rm div}  u \|_{H^4}+\|\rho\|_{H^4}\|{\rm div}  u \|_{L^{\infty}} \\
 \lesssim
 &
 \|\nabla \rho\|_{H^1}\|{\rm div} u\|_{H^4}
 +
 \|\rho\|_{H^4}\|\nabla{\rm div} u\|_{H^1}\\
 \lesssim
 &
 \|\rho\|_{H^4}
 \|{\rm div}  u \|_{H^4},\\
 \|S^1\|_{H^3}
 \leq
 &
 \|S^{11}\|_{H^3}+\|S^{12}\|_{H^3}\\
 \lesssim
 &
 \|u\|_{L^{\infty}}\| \nabla \rho\|_{H^3} +\|u\|_{H^3}\|\nabla \rho\|_{L^\infty}+\|S^{12}\|_{H^3}\\
 \lesssim
 &
 \|u\|_{H^3}\|\nabla \rho\|_{H^3} + \|\rho\|_{H^3} \|{\rm div}  u \|_{H^3} .
\end{split}
\end{equation*}
By using the H$\ddot{o}$lder inequality and Lemma \ref{Lemma 5.4}, we have
\begin{equation*}
\begin{split}
 \|S^1\|_{L^\frac{6}{5}}
 \leq
 &
 \|S^{11}\|_{L^\frac{6}{5}}+\|S^{12}\|_{L^\frac{6}{5}}\\
 \lesssim
 &
 \|u\|_{L^2}\| \nabla \rho\|_{L^3} +\|\rho\|_{L^2}\|\nabla u\|_{L^3}\\
 \lesssim
 &
 \|u\|_{L^2}\|\nabla \rho\|_{H^1} + \|\rho\|_{L^2} \|\nabla u \|_{H^1}.
\end{split}
\end{equation*}
To handle the term of $\|S^2\|_{H^3}$, by using \eqref{3.65} and Lemmas \ref{Lemma 5.4}-\ref{Lemma 5.6}, we have the following inequalities.
\begin{equation*}
\begin{split}
 \|S^2\|_{H^3}
  \lesssim
   &
    \| u \|_{L^{\infty}} \|\nabla  u \|_{H^3}
    +
    \| u \|_{H^3} \|\nabla  u \|_{L^{\infty}}
    +
    \|g(\rho)\|_{L^\infty}
    \|\nabla \rho\|_{H^3}\\
   &
    +
    \|g(\rho)\|_{H^3}
    \|\nabla \rho\|_{L^\infty}
    +
    \|h(\rho)\|_{L^{\infty}}
    \Big(
    \|\theta\|_{L^\infty} \|\nabla \rho\|_{H^3}
    +
    \|\theta\|_{H^3} \|\nabla \rho\|_{L^\infty}
    \Big)\\
   &
    +
    \|h(\rho)\|_{H^3}
    \|\theta\|_{L^\infty}\|\nabla \rho\|_{L^\infty}
    +
    \|g(\rho) \|_{L^{\infty}}\|\nabla^2  u \|_{H^3}\\
   &
    +
    \|g(\rho)\|_{H^3} \|\nabla^2 u\|_{L^{\infty}}
    +
    \|g(\rho)\|_{L^{\infty}} \|\nabla j_0\|_{H^3}
    +
    \|g(\rho)\|_{H^3} \|\nabla j_0\|_{L^{\infty}}
    \\
 \lesssim
   &
    \|u\|_{H^3}\|\nabla u\|_{H^3}
    +
    \|\theta\|_{H^3}\|\nabla\rho\|_{H^3}\\
   &
    +
    \|\rho\|_{H^3}
    \left(
    \|\nabla \rho\|_{H^3}
    +
    \|\nabla  u \|_{H^4}
    +
    \|\nabla \theta\|_{H^3}
    +
    \|\nabla j_0\|_{H^3}
    \right).
\end{split}
\end{equation*}
Similarly,
\begin{equation*}
\begin{split}
 \|S^3\|_{H^3}
 +
  \|S^4\|_{H^3}
 \lesssim
   &
    \|\theta\|_{H^3}\|\nabla u\|_{H^3}
    +
    \|\rho\|_{H^3}\|\nabla^2\theta\|_{H^3}
    +
    \|\nabla\theta\|_{H^1}\|\theta\|_{H^3}\\
   &
    +
    \|\nabla\rho\|_{H^1}
    \left(
    \|\theta\|_{H^3}
    +
    \|j_0\|_{H^3}
    \right)\\
   &
    +
    \left(
    \|\nabla\theta\|_{H^1}
    +
    \|\nabla j_0\|_{H^1}
    \right)
    \|\nabla\rho\|_{H^3}
    +
    \|\nabla  u \|_{H^3}^2\\
   &
    +
    \|u\|_{H^3}
    \left(
    \|\nabla \theta\|_{H^3}
    +
    \|\nabla j_0\|_{H^3}
    \right),
\end{split}
\end{equation*}
and
\begin{equation*}
\begin{split}
\|(S^2,S^3,S^4)\|_{L^\frac{6}{5}}
\lesssim
\|(\rho,u,\theta)\|_{L^2}
\Big(
\|\theta\|_{H^1}
+
\|\nabla(\rho,u,\theta,j_0)\|_{H^1}
+
\|\nabla^2(u,\theta)\|_{H^1}
\Big).
\end{split}
\end{equation*}
Moreover, by using Lemmas \ref{Lemma 5.4} and \ref{Lemma 5.5}, one has
\begin{equation*}
 \|(\nabla u)^T \cdot\nabla\rho\|_{H^3}
 \lesssim
 \|\nabla \rho\|_{L^{\infty}} \|\nabla  u \|_{H^3}
 +
 \|\nabla \rho\|_{H^3} \|\nabla  u \|_{L^{\infty}}
 \lesssim
 \|\nabla \rho\|_{H^3}
 \|\nabla  u \|_{H^3}.
\end{equation*}
And it follows from Lemma \ref{Lemma 5.1} that
\begin{equation*}
\begin{split}
\int_0^t
\|\nabla^2 \rho\|_{L^\infty}^2
\|u^S\|_{H^3}^2
{\rm d}\tau
\leq
\sup\limits_{0\leq \tau\leq t}
\|u(\tau)\|_{H^3}
\int_0^t
\|\nabla^2\rho(\tau)\|_{H^2}^2
{\rm d}\tau,
\end{split}
\end{equation*}
and
\begin{equation*}
\begin{split}
\int_0^t
\|\nabla u(\tau)\|_{L^\infty}
\|\nabla \rho^S(\tau)\|_{H^3}^2
{\rm d}\tau
\leq
\sup\limits_{0\leq \tau\leq t}
\|\nabla u(\tau)\|_{H^2}
\int_0^t
\|\nabla\rho(\tau)\|_{H^3}^2
{\rm d}\tau.
\end{split}
\end{equation*}
By using Lemmas \ref{Lemma 5.3}-\ref{Lemma 5.4} and the Young inequality for series convolution, we get
\begin{equation}\label{3.68}
\begin{split}
&
\int_0^t
\|\nabla u\|_{L^\infty}^2
\sum\limits_{k>k_1}
 (1+2^{6k})
 \left(
 \sum\limits_{l\geq k-1}
 2^{k-l}
 \|\nabla \rho_l\|_{L^2}
 \right)^2
 {\rm d}\tau\\
\lesssim
&
\sup\limits_{0\leq \tau\leq t}
 \|\nabla u(\tau)\|_{H^2}
 \int_0^t
  \sum\limits_{k>k_1}
 (1+2^{6k})
  \left(
 \sum\limits_{l\geq k-1}
 2^{k-l}
 \|\nabla \rho_l\|_{L^2}
 \right)^2
 {\rm d}\tau\\
\lesssim
&
\sup\limits_{0\leq \tau\leq t}
 \|\nabla u(\tau)\|_{H^2}^2
 \int_0^t
  \sum\limits_{k>k_1}
  \left(
 \sum\limits_{l\geq k-1}
 2^{k-l}
 \|\nabla \rho_l\|_{L^2}
 \right)^2
 {\rm d}\tau\\
 &
 +
\sup\limits_{0\leq \tau\leq t}
 \|\nabla u(\tau)\|_{H^2}^2
 \int_0^t
  \sum\limits_{k>k_1}
  \left(
 \sum\limits_{l\geq k-1}
 2^{4(k-l)}2^{3l}
 \|\nabla \rho_l\|_{L^2}
 \right)^2
 {\rm d}\tau\\
\lesssim
&
\sup\limits_{0\leq \tau\leq t}
 \|\nabla u(\tau)\|_{H^2}^2
 \int_0^t
  \sum\limits_{k\in \mathbb Z}
 \|\nabla \rho_k\|_{L^2}^2
 {\rm d}\tau\\
 &
 +
\sup\limits_{0\leq \tau\leq t}
 \|\nabla u(\tau)\|_{H^2}^2
 \int_0^t
  \sum\limits_{k\in \mathbb{Z}}
  2^{6k}
 \|\nabla \rho_k\|_{L^2}^2
 {\rm d}\tau\\
\lesssim
&
\sup\limits_{0\leq \tau\leq t}
 \|\nabla u(\tau)\|_{H^2}^2
 \int_0^t
 \|\nabla \rho\|_{H^3}^2
 {\rm d}\tau.
 \end{split}
\end{equation}
We end up all of the estimates with
\begin{equation*}
 N(t) \leq C\left( N(0) +N^\frac{3}{2}(t)+ N^2(t)\right).
\end{equation*}
This allows to close the estimates globally provided $N(0)$ is small enough. Then, we complete the proof of Proposition \ref{Proposition 2.2}. The existence and uniqueness of solutions are a direct consequence of Proposition \ref{Proposition 2.1} and Proposition \ref{Proposition 2.2} by the standard continuity argument.
$\hfill{\square}$

\section{Decay rates}

\begin{proposition}[Large time behavior]\label{Proposition 2.3}
Under the assumptions of Proposition \ref{Proposition 2.2}, if the initial data satisfies an additional condition that $\|(\rho^0, u^0,\theta^0,j_0^0)\|_{L^1(\mathbb R^3)}<+\infty$, there is a constant $ \tilde C_1>0$ such that for any $t\in [0,T]$, the global solution $(\rho, u,\theta,j_0)(x,t)$ achieved above enjoys the following decay properties.
\begin{equation*}
\begin{split}
\|\nabla^k(\rho, u, \theta,j_0)(t)\|_{L^2(\mathbb{R}^3)}
 \leq
 &
   \tilde C_1(1+t)^{-\frac{3}{4}-\frac{k}{2}},\ \ \ \ for\ k=0,1,2,\\[2mm]
\|\nabla^k(\rho, u, \theta,j_0)(t)\|_{L^2(\mathbb{R}^3)}
 \leq
 &
   \tilde C_1(1+t)^{-\frac{7}{4}},\ \ \ \ \ \ \ for\ k=3,4,\\[2mm]
\|\partial_t(\rho, u)(t)\|_{L^2(\mathbb{R}^3)}
 \leq
 &
   \tilde C_1(1+t)^{-\frac{5}{4}},\\[2mm]
\|\partial_t(\theta,j_0)(t)\|_{L^2(\mathbb{R}^3)}
 \leq
 &
   \tilde C_1(1+t)^{-\frac{3}{4}},\\[2mm]
\|\nabla^k (b^\prime(1)\theta-j_0)\|_{L^2(\mathbb{R}^3)}
 \leq
 &
   \tilde C_1(1+t)^{-\frac{3}{4}-\frac{k+1}{2}},\ \  for\ k=0,1,2.
\end{split}
\end{equation*}
\end{proposition}

\subsection{Energy estimates of the short wave part}

The estimate of the solution in the short wave part is stated as the following proposition.
\begin{proposition}\label{Proposition 4.1}
The following inequality holds true
\begin{equation}\label{4.1}
\begin{split}
&
\|(\rho ,  u ,\theta , j_0)^S(t)\|_{\dot B_{2,2}^3}^2
+
\|(\rho ,  u ,\theta , j_0)^S(t)\|_{\dot B_{2,2}^4}^2\\
\lesssim
&
{\rm e}^{-C_6 t}
\Big(
\|(\rho ,  u ,\theta , j_0)^S(0)\|_{\dot B_{2,2}^3}^2
+
\|(\rho ,  u ,\theta , j_0)^S(0)\|_{\dot B_{2,2}^4}^2
\Big)\\
   &
   +
   \delta
   \int_0^t{\rm e}^{-C_6(t-\tau)}
    \|(\rho , u , \theta, j_0)^L (\tau)\|_{\dot B_{2,2}^3}^2
    {\rm d}\tau\\
   &
   +
   \delta
   \int_0^t{\rm e}^{-C_6(t-\tau)}
   \Big(
    \|(\rho , u, \theta, j_0 )^L(\tau)\|_{\dot B_{2,2}^4}^2
    +
    \|(u,\theta)^L (\tau)\|_{\dot B_{2,2}^5}^2
   \Big)
    {\rm d}\tau,
\end{split}
\end{equation}
where the positive constant $C_6$ is independent of $\delta$.
\end{proposition}

\begin{proof}
Multiplying \eqref{3.27} by $2^{6k}$, we get for $k>k_1>0$
\begin{equation}\label{4.2}
\begin{split}
   &
   \frac{\rm d}{{\rm d}t}
   2^{6k}\mathcal{H}_{h,k}(t)
    +
    C_32^{6k}
    \|(\theta_k,j_{0,k})(t)\|_{L^2}^2
    +
    \frac{C_3}{2}2^{2k_1}2^{6k}
    \|(\rho_k,u_k)(t)\|_{L^2}^2\\
 &
   +
   \frac{C_3}{2}
   2^{8k}
    \|(\rho_k,u_k)(t)\|_{L^2}^2
    +
    C_3
   2^{8k}
    \|(\theta_k,j_{0,k})(t)\|_{L^2}^2
   +
     C_3
     2^{10k}
    \|(u_k,\theta_k,j_{0,k})(t)\|_{L^2}^2\\[1mm]
\lesssim
   &
   2^{6k}\|( S^1_k,S^2_k,S^3_k,S^4_k,\nabla S^{12}_k)(t)\|_{L^2}^2
   +
   2^{6k}\|((\nabla u)^T \cdot\nabla\rho)_k(t)\|_{L^2}^2
      +
 2^{6k}
\|\nabla u\|_{L^\infty}
\|\nabla \rho_k\|_{L^2}^2\\[2mm]
   &
+
2^{6k}
\|\nabla ^2\rho\|_{L^\infty}^2
\|u_k\|_{L^2}^2
+
2^{6k}
\|\nabla u\|_{L^\infty}^2
\left(
 \sum\limits_{l\geq k-1}
 2^{k-l}
 \|\nabla \rho_l\|_{L^2}\right)^2.
\end{split}
\end{equation}
The $l^2$ summation over $k$ for from $k=k_1+1$ to $\infty$ in \eqref{4.2} yields the following inequality
\begin{equation}\label{4.3}
\begin{split}
   &
   \frac{\rm d}{{\rm d}t}
   \sum\limits_{k> k_1}
    2^{6k}\mathcal{H}_{h,k}(t)
    +
    \|(\rho,u,\theta,j_0)^S(t)\|_{\dot B_{2,2}^3}^2\\
   &
   +
   \|(\rho,u,\theta,j_0)^S(t)\|_{\dot B_{2,2}^4}^2
   +
   \|(u ,\theta , j_0)^S(t)\|_{\dot B_{2,2}^5}^2\\[1mm]
\lesssim
   &
   \|(S^1,S^2,S^3,S^4,\nabla S^{12})^S(t)\|_{\dot B_{2,2}^3}^2
   +
   \|((\nabla u)^T \cdot\nabla\rho)^S(t)\|_{\dot B_{2,2}^3}^2
      +
\|\nabla u\|_{L^\infty}
\|\nabla \rho^S\|_{\dot B_{2,2}^3}^2\\[2mm]
   &
+
\|\nabla ^2\rho\|_{L^\infty}^2
\|u^S\|_{\dot B_{2,2}^3}^2
+
\|\nabla u\|_{L^\infty}^2
\sum\limits_{k>k_1}
2^{6k}
\left(
 \sum\limits_{l\geq k-1}
 2^{k-l}
 \|\nabla \rho_l\|_{L^2}\right)^2.
\end{split}
\end{equation}
By Lemma \ref{Lemma 5.1}, Lemmas \ref{Lemma 5.4} and \ref{Lemma 5.5} and the assumption \eqref{2.3}, we have
\begin{equation}\label{4.4}
\begin{split}
  \|(\nabla S^{12})^S(t)\|_{\dot B_{2,2}^3}
  \lesssim
  &
  \|\nabla^4 ( \rho {\rm div}  u )\|_{L^2}\\
\lesssim
  &
  \| \rho  \|_{L^{\infty}}\|\nabla^4{\rm div}  u \|_{L^2}
  +
  \|\nabla^4 \rho \|_{L^2}\|{\rm div}  u \|_{L^{\infty}} \\
\lesssim
  &
  \delta\big(\|\rho \|_{\dot B_{2,2}^4}
             +
             \|u \|_{\dot B_{2,2}^5}\big),
\end{split}
\end{equation}
\begin{equation}\label{4.5}
\begin{split}
  \|(S^1)^S(t)\|_{\dot B_{2,2}^3}
\lesssim
&
  \|\nabla^3( u \cdot \nabla \rho )\|_{L^2}
  +
  \|\nabla^3 S^{12}\|_{L^2}\\
\lesssim
  &
  \|u\|_{L^{\infty}}\|\nabla^4 \rho \|_{L^2}
  +
  \|\nabla^3  u \|_{L^2}\|\nabla \rho \|_{L^\infty}+\|\nabla^3 S^{12}\|_{L^2}\\
\lesssim
  &
   \delta
   \big(
   \|(\rho , u )\|_{\dot B_{2,2}^3}
   +
   \|(\rho , u )\|_{\dot B_{2,2}^4}
   \big).
\end{split}
\end{equation}
To handle the term $\|(S^2)^S\|_{\dot B_{2,2}^3}$, by Lemma \ref{Lemma 5.1}, Lemmas \ref{Lemma 5.4}-\ref{Lemma 5.6} and the assumption \eqref{2.3}, the following inequality holds true.
\begin{eqnarray}\label{4.6}
\begin{split}
  \|(S^2)^S(t)\|_{\dot B_{2,2}^3}
\lesssim
  &
  \| u \|_{L^{\infty}} \|\nabla^4  u \|_{L^2}
  +
  \|\nabla^3  u \|_{L^2} \|\nabla  u \|_{L^{\infty}}\\
  &
  +
  \|g(\rho)\|_{L^{\infty}} \|\nabla^4 \rho \|_{L^2}
  +
  \|\nabla^3 g(\rho)\|_{L^2}\|\nabla \rho\|_{L^\infty}\\
  &
  +
  \|h(\rho)\|_{L^\infty}
  \Big(
  \|\theta\|_{L^\infty}\|\nabla^4 \rho\|_{L^2}
  +
  \|\nabla^3\theta\|_{L^2}\|\nabla \rho\|_{L^\infty}
  \Big)\\
  &
  +
  \|\nabla^3 h(\rho)\|_{L^2}
  \|\theta\|_{L^\infty}\|\nabla \rho\|_{L^\infty}
  \\
  &
  +
  \|g(\rho)\|_{L^{\infty}} \|\nabla^5 u \|_{L^2}
  +
  \|\nabla^3 g(\rho)\|_{L^2}\|\nabla^2 u\|_{L^\infty}\\
  &
  +
  \|g(\rho)\|_{L^{\infty}} \|\nabla^4 j_0 \|_{L^2}
  +
  \|\nabla^3 g(\rho)\|_{L^2}\|\nabla j_0\|_{L^\infty}\\
\lesssim
  &
  \delta\big(
  \|(\rho,u,\theta)\|_{\dot B_{2,2}^3}
  +
  \|(\rho,u,j_0)\|_{\dot B_{2,2}^4}
  +
  \|u\|_{\dot B_{2,2}^5}
  \big).
\end{split}
\end{eqnarray}
Similarly,
\begin{equation}
\begin{split}
  \|(S^3)^S(t)\|_{\dot B_{2,2}^3}
   \lesssim
  &
   \delta
   \big(
   \|(\rho,\theta,j_0)\|_{\dot B_{2,2}^3}
   +
   \|(u,\theta,j_0)\|_{\dot B_{2,2}^4}
   +
   \|\theta\|_{\dot B_{2,2}^5}
   \big).
\end{split}
\end{equation}
Moreover, one has
\begin{equation}\label{4.8}
\begin{split}
  \|((\nabla u)^T \cdot\nabla\rho)^S(t)\|_{\dot B_{2,2}^3}
\lesssim
  &
  \|\nabla \rho \|_{L^{\infty}} \|\nabla^4  u \|_{L^2}
  +
  \|\nabla^4 \rho \|_{L^2} \|\nabla  u \|_{L^{\infty}}\\
\lesssim
  &
  \delta
  \|(\rho,u) \|_{\dot B_{2,2}^4},
\end{split}
\end{equation}
\begin{equation}
\|\nabla u \|_{L^{\infty}}\|(\nabla\rho)^S(t)\|_{\dot B_{2,2}^3}^2
 \lesssim
  \delta
   \|\rho\|_{\dot B_{2,2}^4}^2,
\end{equation}
and
\begin{equation*}
\|\nabla^2\rho\|_{L^\infty}^2
\|u^S\|_{\dot B^3_{2,2}}^2
\leq
\|\nabla^2\rho\|_{H^2}^2
\|u\|_{\dot B^3_{2,2}}^2
\leq
\delta^2
\|u\|_{\dot B^3_{2,2}}^2.
\end{equation*}
Similar to the estimate \eqref{3.68}, by using Lemmas \ref{Lemma 5.3}-\ref{Lemma 5.4}, the Young inequality for series convolution and the assumption \eqref{2.3}, we get
\begin{equation}\label{4.10}
\|\nabla u\|_{L^\infty}^2
\sum\limits_{k>k_1}
2^{6k}
\left(
 \sum\limits_{l\geq k-1}
 2^{k-l}
 \|\nabla \rho_l\|_{L^2}\right)^2
 \lesssim
 \|\nabla u\|_{H^2}^2
 \|\nabla \rho\|_{\dot B^3_{2,2}}^2
\lesssim
\delta^2
  \|\nabla \rho\|_{\dot B_{2,2}^3}^2.
\end{equation}
Substituting \eqref{4.4}-\eqref{4.10} into \eqref{4.3} yields
\begin{equation}
\begin{split}
   &
   \frac{\rm d}{{\rm d}t}
   \sum\limits_{k>k_1}
    2^{6k}\mathcal{H}_{h,k}(t)
   +
    \|(\rho,u,\theta,j_0)^S(t)\|_{\dot B_{2,2}^3}^2\\
   &
   +
   \|(\rho,u,\theta,j_0)^S(t)\|_{\dot B_{2,2}^4}^2
   +
   \|(u ,\theta , j_0)^S(t)\|_{\dot B_{2,2}^5}^2\\[1mm]
\lesssim
   &
   \delta
   \big(
  \|(\rho , u , \theta, j_0)(t)\|_{\dot B_{2,2}^3}^2
  +
  \|(\rho , u, \theta, j_0 )(t)\|_{\dot B_{2,2}^4}^2
  +
  \|(u,\theta)(t) \|_{\dot B_{2,2}^5}^2
  \big).
\end{split}
\end{equation}
By using the decomposition \eqref{5.3} in Lemma \ref{Lemma 5.2} and choosing the parameter $\delta$ is suitably small, we obtain
\begin{equation}\label{4.12}
\begin{split}
   &
   \frac{\rm d}{{\rm d}t}
   \sum\limits_{k> k_1}
    2^{6k}\mathcal{H}_{h,k}(t)
   +
    \frac{1}{2}
    \|(\rho,u,\theta,j_0)^S(t)\|_{\dot B_{2,2}^3}^2\\
   &
   +
    \frac{1}{2}
   \|(\rho,u,\theta,j_0)^S(t)\|_{\dot B_{2,2}^4}^2
   +
    \frac{1}{2}
   \|(u ,\theta , j_0)^S(t)\|_{\dot B_{2,2}^5}^2\\[1mm]
\lesssim
   &
\delta
   \big(
  \|(\rho , u , \theta, j_0)^L(t)\|_{\dot B_{2,2}^3}^2
  +
  \|(\rho , u, \theta, j_0 )^L(t)\|_{\dot B_{2,2}^4}^2
  +
  \|(u,\theta)^L(t)\|_{\dot B_{2,2}^5}^2
  \big).
\end{split}
\end{equation}
It follows from \eqref{3.25} that
\begin{equation}\label{4.13}
\sum\limits_{k>k_1}
    2^{6k}\mathcal{H}_{h,k}(t)
\sim
\|(\rho ,  u ,\theta , j_0)^S(t)\|_{\dot B_{2,2}^3}^2
+
\|(\rho ,  u ,\theta , j_0)^S(t)\|_{\dot B_{2,2}^4}^2,
\end{equation}
for any $0\leq t\leq T$.
Then, from \eqref{4.12} and \eqref{4.13}, there exists a positive constant $C_6$ independent of $\delta$ such that
\begin{equation}\label{4.14}
\begin{split}
   &
   \frac{\rm d}{{\rm d}t}
   \sum\limits_{k>k_1}
    2^{6k}\mathcal{H}_{h,k}(t)
   +
   C_6
   \sum\limits_{k>k_1}
    2^{6k}\mathcal{H}_{h,k}(t)\\
\lesssim
   &
\delta
  \|(\rho , u , \theta, j_0)^L(t)\|_{\dot B_{2,2}^3}^2
  +
  \|(\rho , u, \theta, j_0 )^L(t)\|_{\dot B_{2,2}^4}^2
  +
  \|(u,\theta)^L(t)\|_{\dot B_{2,2}^5}^2
  \big).
\end{split}
\end{equation}
Multiplying \eqref{4.14} with ${\rm e}^{C_6t}$ and integrating with respect to $t$ over $[0,t]$, we have
\begin{equation} \label{4.15}
\begin{split}
   \sum\limits_{k>k_1}
    2^{6k}\mathcal{H}_{h,k}(t)
   \lesssim
   &
   {\rm e}^{-C_6 t}
   \sum\limits_{k>k_1}
    2^{6k}\mathcal{H}_{h,k}(0)
   +
   \delta
   \int_0^t{\rm e}^{-C_6(t-\tau)}
    \|(\rho , u , \theta, j_0)^L(\tau)\|_{\dot B_{2,2}^3}^2
    {\rm d}\tau\\
   &
   +
   \delta
   \int_0^t{\rm e}^{-C_6(t-\tau)}
   \Big(
  \|(\rho , u, \theta, j_0 )^L(\tau)\|_{\dot B_{2,2}^4}^2
  +
  \|(u,\theta)^L(\tau)\|_{\dot B_{2,2}^5}^2
   \Big)
    {\rm d}\tau.
\end{split}
\end{equation}
Since $\sum\limits_{k>k_1}2^{6k}\mathcal{H}_{h,k}(0)\sim \|(\rho_0,u_0,\theta_0,j_0^0)^S\|_{\dot B_{2,2}^3}+\|(\rho_0,u_0,\theta_0,j_0^0)^S\|_{\dot B_{2,2}^4}$, from \eqref{4.13} and \eqref{4.15}, we get \eqref{4.1}.
\end{proof}

\subsection{Decay rates of the long wave part}
In this subsection, based on the $L^2$-norm decay estimates for Fourier analysis on the
linearized system, we obtain decay estimates of the long wave parts of solutions to \eqref{3.1}.
Let $\mathbb{A}$ be the following matrix of the differential operators of the form
\begin{equation*}
\mathbb{A}
 =
  \left(
           \begin{array}{cccc}
             0 & {\rm div}& 0& 0 \\
              \nabla & -\mu\Delta-(\mu+\lambda)\nabla{\rm div}&\nabla&\frac{1}{3\mathcal{C}}\nabla \\
              0& \frac{2}{3}{\rm div}& -\frac{2}{3}k\Delta+\frac{2\gamma}{3}& -\frac{2}{3}b \\
               0& 0& -\mathcal{C}\gamma & -a\Delta + \mathcal{C}b\\
           \end{array}
         \right),
\end{equation*}
and
\begin{equation*}
\overline{\mathbb{U}}_k(t):=(\overline{\rho}_k(t), \overline{u}_k(t),\overline{\theta}_k(t),\overline{j_{0,k}}(t))^T,\ \
{\rm and}\ \
\mathbb{U}_k(0)
 =
  (\rho^0_k, u^0_k,\theta^0_k,j_{0,k}^0)^T.
\end{equation*}
Applying the homogeneous frequency localized operator $\dot\Delta_k$ to \eqref{3.1} yields for all $k\in \mathbb{Z}$
we have the following corresponding linearized problem
\begin{equation}\label{4.16}
\left\{
\begin{array}{lll}
\partial_t\overline{\mathbb{U}}_k+\mathbb{A}\overline{\mathbb{U}}_k=0,\ \ \ \ {\rm for}\ t>0,\\
\overline{\mathbb{U}}_k|_{t=0}=\mathbb U_k(0).
\end{array}
\right.
\end{equation}
Applying the Fourier transform on \eqref{4.16} with respect to the $x$-variable and solving the ordinary
differential equation with respect to $t$, we have
\begin{equation*}
\overline{\mathbb{U}}_k(t)
=
  \mathcal{A}(t) \mathbb{U}_k(0),
\end{equation*}
where $\mathcal{A}(t)={\rm e}^{-t\mathbb{A}}(t\geq 0)$ is the semigroup generated by the linear operator $\mathbb{A}$ and $\mathcal{A}(t)g:=
\mathcal{F}^{-1}({\rm e}^{-t\mathbb{A}_\xi} \hat{g}(\xi))$ with
\begin{equation*}
\mathbb{A}_\xi
 =
  \left(
           \begin{array}{cccc}
            0 & i\xi^T& 0& 0 \\
              i\xi& \nu|\xi|^2\delta_{ij}+\xi_i\xi_j&i\xi&\frac{1}{3\mathcal{C}}i\xi \\
              0& \frac{2}{3}i\xi^T& \frac{2}{3}k|\xi|^2+\frac{2\gamma}{3}& -\frac{2}{3}b \\
               0& 0& -\mathcal{C}\gamma & a|\xi|^2 + \mathcal{C}b\\
           \end{array}
         \right).
\end{equation*}

\begin{lemma}\label{Lemma 4.1}
Let $m\geq 0$ be an integer and $1\leq p\leq 2$, then for any given $k\leq k_1$, it holds that
\begin{equation*}
\|\nabla^{m}\big(\mathcal{A}(t)\mathbb{U}_k(0))\|_{L^2}
\leq
C(1+t)^{-\frac{3}{4}(\frac{1}{p}-\frac{1}{2})-\frac{m}{2}}
\|\mathbb{U}(0)\|_{L^p}.
\end{equation*}
\end{lemma}

\begin{proof}
For the linearized system of \eqref{3.1}, similar to the estimate \eqref{3.46}, we have
\begin{equation}
\frac{\rm d}{{\rm d}t}
 \mathcal{L}_l(t,\xi)
  +
  C_4 |\xi|^2
   \mathcal{L}_l(t,\xi)
    \leq
     0,
\end{equation}
which implies that for $|\xi|\leq r_0$
\begin{equation}\label{4.18}
\mathcal{L}_l(t,\xi)
 \lesssim
  {\rm e}^{-C_4|\xi|^2 t}
   \mathcal{L}_l(0,\xi).
\end{equation}
By using the Plancherel theorem, \eqref{3.49}, \eqref{3.55} and \eqref{4.18}, we have for $k\leq k_1$
\begin{equation}\label{4.19}
\begin{split}
\|\partial_x^\alpha(\overline{\rho}_k,\overline{d}_k,\overline{\theta}_k,\overline{j_{0,k}})(t)\|
=
&
\|(i\xi)^\alpha(\widehat{\overline{\rho} _k},\widehat{\overline{d}_k},\widehat{\overline{\theta}_k},\widehat{\overline{j_{0,k}}})\|_{L_\xi^2}\\
=
  &
   \left(
     \int_{\mathbb{R}^3}
      \left|
       (i\xi)^\alpha
        (\widehat{\overline{\rho} _k},\widehat{\overline{d}_k},\widehat{\overline{\theta}_k},\widehat{\overline{j_{0,k}}})(\xi,t)
      \right|^2
     {\rm d}\xi
   \right)^\frac{1}{2}\\
 \leq
  &
   C\left(
     \int_{|\xi|\leq R_0}
     |\xi|^{2|\alpha|}
         |(\widehat{\overline{\rho} _k},\widehat{\overline{d}_k},\widehat{\overline{\theta}_k},\widehat{\overline{j_{0,k}}})(\xi,t)|^2
     {\rm d}\xi
   \right)^\frac{1}{2}\\
 \leq
  &
   C\left(
     \int_{|\xi|\leq r_0}
       |\xi|^{2|\alpha|}
       {\rm e}^{-C|\xi|^2t}
         |(\widehat{\rho_k},\widehat{d_k},\widehat{\theta_k},\widehat{j_{0,k}})(\xi,0)|^2
     {\rm d}\xi
   \right)^\frac{1}{2}\\
  &
  +
   C\left(
     \int_{r_0\leq |\xi|\leq R_0}
       |\xi|^{2|\alpha|}
       {\rm e}^{-\iota t}
         |(\widehat{\rho_k},\widehat{d_k},\widehat{\theta_k},\widehat{j_{0,k}})(\xi,0)|^2
     {\rm d}\xi
   \right)^\frac{1}{2}.\\
\end{split}
\end{equation}
From \eqref{4.19} and the Hausdorff-Young inequality, for $k\leq k_1$, we have
\begin{equation}\label{4.20}
\begin{split}
\|\partial_x^\alpha(\overline{\rho}_k,\overline{d}_k,\overline{\theta}_k,\overline{j_{0,k}})(t)\|_{L^2}
 \leq
  &
   C
    \|(\widehat{\rho _k},\widehat{d_k},\widehat{\theta_k},\widehat{j_{0,k}})(0)\|_{L_\xi^{\frac{p}{p-1}}}
     (1+t)^{-\frac{3}{4}(\frac{1}{p}-\frac{1}{2})-\frac{|\alpha|}{2}}\\
 \leq
  &
   C
    \|(\rho ,u,\theta ,j_{0})(0)\|_{L^p}
     (1+t)^{-\frac{3}{4}(\frac{1}{p}-\frac{1}{2})-\frac{|\alpha|}{2}}.
\end{split}
\end{equation}
Similar as the estimates \eqref{4.20}, we get for any $k\leq k_1$
\begin{equation}\label{4.21}
\begin{split}
\|\partial_x^\alpha(\overline{\mathcal{P}  u})_k(t)\|_{L^2}
 \leq
   C
    \|u(0)\|_{L^p}
     (1+t)^{-\frac{3}{4}(\frac{1}{p}-\frac{1}{2})-\frac{|\alpha|}{2}}.
\end{split}
\end{equation}
Then, from \eqref{4.20} and \eqref{4.21}, we complete the proof of Lemma \ref{Lemma 4.1}.
\end{proof}

In what follows, based on the estimates in Lemma \ref{Lemma 4.1}, we establish the time decay estimates on the long wave part of the solution to the nonlinear problem \eqref{3.1}. Denote
\begin{equation*}
\mathbb{U}_k(t):=(\rho _k(t), u _k(t),\theta_k(t),j_{0,k}(t))^T,
\end{equation*}
for any $k\leq k_1$.
Then, from \eqref{3.1}, we have
\begin{equation}\label{4.22}
\left\{
\begin{array}{lll}
\partial_t\mathbb{U}_k+\mathbb{A}\mathbb{U}_k=S(\mathbb{U}_k),\ \ \ \ {\rm for}\ t>0,\\
\mathbb{U}_k|_{t=0}=\mathbb U_k(0),
\end{array}
\right.
\end{equation}
where
\begin{equation*}
S(\mathbb{U}_k)
 =
  (S^1_k,S^2_k,S^3_k,S^4_k)^T.
\end{equation*}
By Duhamel's principle, we rewrite the solution of \eqref{4.22} as follows
\begin{eqnarray}\label{4.23}
\begin{split}
\mathbb{U}_k(t)
=
 &
  \mathcal{A}(t) \mathbb{U}_k(0)
  +\int_0^t \mathcal{A}(t-\tau)S(\mathbb{U}_k)(\tau){\rm d}\tau.
\end{split}
\end{eqnarray}

\begin{proposition}\label{Proposition 4.2}
For any integer $m \geq 0$, there exists a positive constant $C_7$ depending on $k_1$, such that
\begin{equation*}
\begin{split}
 \|(\rho , u ,\theta , j_0)^L(t)\|_{\dot B_{2,2}^m}
\leq
 &
  C_7(1+t)^{-\frac{3}{4}-\frac{m}{2}}
  \|(\rho , u ,\theta , j_0)(0)\|_{L^1}\\
 &
 +
 C_7
 \delta
 \int_0^\frac{t}{2}
 (1+t-\tau)^{-\frac{3}{4}-\frac{m}{2}}
 \|\nabla (\rho , u,\theta,j_0 )(\tau)\|_{L^2}
 {\rm d}\tau\\
 &
 +
 C_7
 \delta
 \int_0^\frac{t}{2}
 (1+t-\tau)^{-\frac{3}{4}-\frac{m}{2}}
 \|\nabla^2(u,\theta )(\tau)\|_{L^2}
 {\rm d}\tau\\
 &
 +
 C_7
 \int_0^\frac{t}{2}
 (1+t-\tau)^{-\frac{3}{4}-\frac{m}{2}}
 \|(\rho,\theta,j_0 )(\tau)\|_{L^2}^2
 {\rm d}\tau\\
 &
 +
 C_7\int_\frac{t}{2}^t
 (1+t-\tau)^{-\frac{m}{2}}
 \|S(\mathbb{U})\|_{L^2}
 {\rm d}\tau,
\end{split}
\end{equation*}
where $S(\mathbb{U}):=(S^1,S^2,S^3,S^4)$.
\end{proposition}

\begin{proof}
From \eqref{4.23}, by using Lemma \ref{Lemma 4.1}, we have, for $k\leq k_1$
\begin{equation}\label{4.24}
\begin{split}
\|\nabla^{m}\mathbb{U}_k(t)\|_{L^2}
\leq
&
C(1+t)^{-\frac{3}{4}-\frac{m}{2}}\|\mathbb{U}(0)\|_{L^1}\\
&
+
C\int_0^\frac{t}{2}
(1+t-\tau)^{-\frac{3}{4}-\frac{m}{2}}\|S(\mathbb{U})(\tau)\|_{L^1}
{\rm d}\tau\\
&
+
C\int_\frac{t}{2}^t
(1+t-\tau)^{-\frac{m}{2}}\|S(\mathbb{U})(\tau)\|_{L^2}
{\rm d}\tau.
\end{split}
\end{equation}
Notice that under the condition \eqref{2.3}, by using the H${\rm \ddot{o}}$lder inequality, we have
\begin{equation}\label{4.25}
\begin{split}
\|S(\mathbb{U})(\tau)\|_{L^1}
 \leq
 &
  C
  \|(\rho , u, \theta, j_0)(\tau)\|_{H^1}
  \big(
  \|\nabla (\rho , u, \theta, j_0 )(\tau)\|_{L^2}
  +
  \|\nabla^2  (u, \theta) (\tau)\|_{L^2}
  \big)\\
  &
  +
  \|\theta\|_{L^2}^2
  +
  \|\rho\|_{L^2}
  \big(
  \|\theta\|_{L^2}
  +
  \|j_0\|_{L^2}
  \big)\\
 \leq
 &
  C
  \delta
  \big(
  \|\nabla (\rho , u, \theta, j_0 )(\tau)\|_{L^2}
  +
  \|\nabla^2  (u, \theta) (\tau)\|_{L^2}
  \big)
 +
 C\|(\rho,\theta,j_0)\|_{L^2}^2.
\end{split}
\end{equation}
Then, putting \eqref{4.25} into \eqref{4.24} and taking a $l^2$ summation over $k$ with $k\leq k_1$, we complete the proof of Proposition \ref{Proposition 4.2}.
\end{proof}

\subsection{Proof of Proposition \ref{Proposition 2.3}}
In this subsection, by combining Proposition \ref{Proposition 4.1} with Proposition \ref{Proposition 4.2}, we get the large time behavior of solution to the nonlinear problem \eqref{3.1}.

\begin{lemma}\label{Lemma 4.2}
Under the assumptions of Proposition \ref{Proposition 2.3}, it holds that
\begin{equation}\label{4.26}
\|\nabla^m(\rho , u,\theta,j_0)(t)\|_{L^2}
\leq
 C(1+t)^{-\frac{3}{4}-\frac{m}{2}},\ \ \ \ for \ m=0,1,2,
\end{equation}
\begin{equation}\label{4.27}
\|\nabla^m(\rho , u,\theta,j_0)(t)\|_{L^2}
\leq
 C(1+t)^{-\frac{7}{4}},\ \qquad for \ m=3,4.\ \ \
\end{equation}
\end{lemma}

\begin{proof}
Denote that
\begin{equation}\label{4.28}
\begin{split}
M(t)
 :=
   \sup\limits_{0\leq \tau \leq t}
   \sum\limits_{m=0}^2
   (1+\tau)^{\frac{3}{4}+\frac{m}{2}}
   \|\nabla^m(\rho , u ,\theta, j_0)(\tau)\|_{L^2}.
\end{split}
\end{equation}
Notice that $M(t)$ is non-decreasing, and for $0\leq m\leq 2$
\begin{equation}\label{4.29}
\|\nabla^m (\rho , u ,\theta, j_0)(\tau)\|_{L^2}^2
 \leq
  C_8
  (1+\tau)^{-\frac{3}{4}-\frac{m}{2}}
  M(t),\ \ 0\leq \tau\leq t,
\end{equation}
holds true for some positive constant $C_8$ independent of $\delta$.

By using the H${\rm \ddot{o}}$lder inequality and \eqref{3.65}, we have
\begin{equation}\label{4.30}
\begin{split}
\|S(\mathbb{U})\|_{L^2}
 \lesssim
 &
  \|(\rho,u,\theta)\|_{L^\infty}
   \|\nabla(\rho,u,\theta,j_0)\|_{L^2}
   +
   \|\rho\|_{L^\infty}
   \|\nabla^2(u,\theta)\|_{L^2}\\
  &
  +
  \|\nabla u\|_{L^\infty}
  \|\nabla u\|_{L^2}
  +
  \|\theta\|_{L^\infty}
  \|\theta\|_{L^2}
  +
  \|\rho\|_{L^\infty}
  \|(\theta,j_0)\|_{L^2}\\
\lesssim
 &
  \|\nabla(\rho,u,\theta)\|_{H^1}
   \|\nabla(\rho,u,\theta,j_0)\|_{L^2}
   +
   \|\nabla\rho\|_{H^1}
   \|\nabla^2(u,\theta)\|_{L^2}\\
  &
  +
  \|\nabla^2 u\|_{H^1}
  \|\nabla u\|_{L^2}
  +
  \|\nabla\theta\|_{H^1}
  \|(\theta,j_0)\|_{L^2}.
\end{split}
\end{equation}
Combining Proposition \ref{Proposition 4.2}, \eqref{4.28} and \eqref{4.30}, we have for $m\geq 0$
\begin{equation}\label{4.31}
\begin{split}
 \|(\rho , u ,\theta , j_0)^L(t)\|_{\dot B_{2,2}^m}
\leq
 &
  C_7
  (1+t)^{-\frac{3}{4}-\frac{m}{2}}
  \|(\rho , u ,\theta , j_0)(0)\|_{L^1}\\
 &
 +
 C_7
 \delta
 M(t)
 \int_0^\frac{t}{2}
 (1+t-\tau)^{-\frac{3}{4}-\frac{m}{2}}
 (1+\tau)^{-\frac{5}{4}}
 {\rm d}\tau\\
 &
 +
 C_7
 \delta
 M(t)
 \int_0^\frac{t}{2}
 (1+t-\tau)^{-\frac{3}{4}-\frac{m}{2}}
  (1+\tau)^{-\frac{7}{4}}
 {\rm d}\tau\\
 &
 +
 C_7
 \delta^\frac{1}{4}
 M^\frac{7}{4}(t)
 \int_0^\frac{t}{2}
 (1+t-\tau)^{-\frac{3}{4}-\frac{m}{2}}
  (1+\tau)^{-\frac{21}{16}}
 {\rm d}\tau\\
 &
 +
 C_7
 \delta^\frac{1}{4}
 M^\frac{7}{4}(t)
 \int_\frac{t}{2}^t
 (1+t-\tau)^{-\frac{m}{2}}
 (1+\tau)^{-\frac{29}{16}}
 {\rm d}\tau.
\end{split}
\end{equation}
From \eqref{4.31}, we arrive at
\begin{equation}\label{4.32}
\begin{split}
&
 \|(\rho , u ,\theta , j_0)^L(t)\|_{\dot B_{2,2}^m}\\[2mm]
\leq
&
\left\{
\begin{array}{llll}
 C_9
 (1+t)^{-\frac{3}{4}-\frac{m}{2}}
 \big(
  \|(\rho , u ,\theta , j_0)(0)\|_{L^1}
 +
 \delta M(t)
 +
 \delta^\frac{1}{4}
 M^\frac{7}{4}(t)
 \big),\ \ {\rm for}\ 0\leq m\leq 2,\\[2mm]
 C_9
 (1+t)^{-\frac{7}{4}}
 \big(
  \|(\rho , u ,\theta , j_0)(0)\|_{L^1}
 +
 \delta M(t)
 +
 \delta^\frac{1}{4}
 M^\frac{7}{4}(t)
 \big),\ \ \ \ \ \ {\rm for}\ m\geq 3,
 \end{array}
 \right.
\end{split}
\end{equation}
where $C_9$ denotes some positive constant independent of $\delta$.
It follows from \eqref{4.1} and \eqref{4.32} that
\begin{equation}\label{4.33}
\begin{split}
 &
\|(\rho ,  u ,\theta , j_0)^S(t)\|_{\dot B_{2,2}^3}^2
+
\|(\rho ,  u ,\theta , j_0)^S(t)\|_{\dot B_{2,2}^4}^2\\
\lesssim
 &
 {\rm e}^{-C_6 t}
 \Big(
 \|(\rho ,  u ,\theta , j_0)^S(0)\|_{\dot B_{2,2}^3}^2
+
\|(\rho ,  u ,\theta , j_0)^S(0)\|_{\dot B_{2,2}^4}^2
 \Big)\\
 &
 +
 \delta
 C_9^2
 \big(
 \|(\rho , u ,\theta , j_0)(0)\|_{L^1}
 +
 \delta M(t)
 +
 \delta^\frac{1}{4}
 M^\frac{7}{4}(t)
 \big)^2
 \int_0^t
 {\rm e}^{-C_6(t-\tau)}
 (1+\tau)^{-\frac{7}{2}}
 {\rm d}\tau\\
 \lesssim
 &
 {\rm e}^{-C_6 t}
 \Big(
 \|(\rho ,  u ,\theta , j_0)^S(0)\|_{\dot B_{2,2}^3}^2
+
\|(\rho ,  u ,\theta , j_0)^S(0)\|_{\dot B_{2,2}^4}^2
\Big)\\[2mm]
 &
 +
 \delta
 C_9^2
 (1+t)^{-\frac{7}{2}}
\big(
 \|(\rho , u ,\theta , j_0)(0)\|_{L^1}
 +
 \delta M(t)
 +
 \delta^\frac{1}{4}
 M^\frac{7}{4}(t)
 \big)^2.
\end{split}
\end{equation}
By using Lemma \ref{Lemma 5.2}, we obtain for $0\leq m\leq 2$
\begin{equation}\label{4.34}
\begin{split}
\|\nabla^m(\rho,u,\theta,j_0)(t)\|_{L^2}
\lesssim
&
 \|(\rho,u,\theta,j_0)^L(t)\|_{\dot{B}^m_{2,2}}
 +
 \|(\rho,u,\theta,j_0)^S(t)\|_{\dot{B}^m_{2,2}}\\
\lesssim
&
 \|(\rho,u,\theta,j_0)^L(t)\|_{\dot{B}^m_{2,2}}
 +
 \|(\rho,u,\theta,j_0)^S(t)\|_{\dot{B}^3_{2,2}}.
\end{split}
\end{equation}
From \eqref{4.32}, \eqref{4.33} and \eqref{4.34}, by noting the definition of $M(t)$ and the smallness of $\delta$, there exists a positive constant $C_{10}$ independent of $\delta$, such that
\begin{equation}\label{4.35}
\begin{split}
M(t)
 \leq
 &
   C_{10}
   \{
   \|(\rho , u ,\theta , j_0)(0)\|_{L^1\cap H^4}
   +
   \delta^\frac{1}{4}
   M^\frac{7}{4}(t)
   \}\\
 \leq
 &
    C_{10}
   \|(\rho , u ,\theta , j_0)(0)\|_{L^1\cap H^4}
   +
   \frac{1}{8}
   C_{10}^8
   +
   \frac{7}{8}
   \delta^\frac{2}{7}
   M^2(t)\\
 :=
 &
  \tilde C_{10}
   +
   \frac{7}{8}
   \delta^\frac{2}{7}
   M^2(t).
\end{split}
\end{equation}
Now we can claim $M(t)\leq C$. Suppose $M(t)>2\tilde C_{10}$ for any $t\in [\bar t,+\infty)$ with a constant $\bar t>0$. Since $M(0)=\|( \rho_0, u_0, \theta_0,j_0^0)\|_{H^2}$ is small (see the assumption \eqref{1.7}) and $M(t)\in C^0[0,+\infty)$, there exists $t_0\in (0,\bar t)$ such that $M(t_0)=2\tilde C_{10}$. From \eqref{4.35}, we have
\begin{equation*}
M(t_0)
 \leq
  \tilde C_{10}
   +
   \frac{7}{8}
   \delta^\frac{2}{7}
   M^2(t_0).
\end{equation*}
By a directly calculation, we have
\begin{equation}\label{4.36}
M(t_0)
 \leq
  \frac{\tilde C_{10}}{1-\frac{7}{8}\delta^\frac{2}{7}M(t_0)}.
\end{equation}
Let $\delta$ be a small constant such that $\frac{7}{8}\delta^\frac{2}{7}< \frac{1}{4\tilde C_{10}}$. Then, from \eqref{4.36}, we get $M(t_0)< 2\tilde C_{10}$. This become a contradiction with the assumption $M(t_0)=2\tilde C_{10}$. So, we have $M(t)\leq 2\tilde C_{10}$ for any $t\in [\bar t,+\infty)$. By using the continuity of $M(t)$, we have $M(t)\leq C$ for any $t\in[0,+\infty)$. By the definition of $M(t)$ in \eqref{4.28}, we prove \eqref{4.26}.
Combining \eqref{4.32} for $m=3,4$ with \eqref{4.33} and using Lemma \ref{Lemma 5.2} and $M(t)\leq C$, we prove \eqref{4.27}.
\end{proof}

By using \eqref{4.26}-\eqref{4.27}, from \eqref{3.1}, we achieve
\begin{equation*}
\begin{split}
\|\partial_t (\rho,u) (t)\|_{L^2}
\lesssim
&
 \|{\rm div} u (t)\|_{L^2}
 +
 \|\nabla(\rho,\theta,j_0) (t)\|_{L^2}
 +
 \| {\rm div}\mathbb T(t)\|_{L^2}\\
 &
 +
 \|(S^1,S^2)(t)\|_{L^2}\\
\lesssim
&
\|\nabla  (\rho,u,j_0) (t)\|_{L^2}
+
\|\nabla^2 u (t)\|_{L^2}\\
\lesssim
&
  (1+t)^{-\frac{5}{4}},
\end{split}
\end{equation*}
and
\begin{equation*}
\begin{split}
\|\partial_t(\theta,j_0)(t)\|_{L^2}
\lesssim
&
 \|{\rm div}u (t)\|_{L^2}
 +
 \|\Delta \theta(t)\|_{L^2}
 +
 \|(\theta,j_0)(t)\|_{L^2}\\
 &
 +
 \|\Delta j_0 (t)\|_{L^2}
 +
 \|(S^3,S^4)(t)\|_{L^2}\\
\lesssim
&
\|\nabla  (u,\theta,j_0) (t)\|_{L^2}
+
\|(\theta,j_0)(t)\|_{L^2}
+
\|\nabla^2 (\theta,j_0) (t)\|_{L^2}\\
\lesssim
 &
  (1+t)^{-\frac{3}{4}}.
\end{split}
\end{equation*}
Next, we show the decay estimates of the combination $\mathcal{L}\sigma_ab^\prime(1)\theta-\mathcal{L}\sigma_aj_0$, i.e. $\gamma\theta-bj_0$.
Set
\begin{equation*}
\Xi
 =
  \gamma \theta
  -
  b j_0,
\end{equation*}
then from $\eqref{3.1}_3$ and $\eqref{3.1}_4$, we have
\begin{equation}\label{4.37}
\partial_t\Xi
+
(\frac{2}{3}+\mathcal{C})\Xi
+
\frac{2}{3}\gamma{\rm div}u
-
\frac{2\kappa}{3}\gamma
\Delta\theta
+
ab\Delta j_0
=
\gamma S^3
+
b S^4.
\end{equation}
Multiplying \eqref{4.37} with $\Xi$ and integrating with respect to $x$ in $\mathbb{R}^3$ and using the Young inequality, we have
\begin{equation}\label{4.38}
\begin{split}
&
\frac{\rm d}{{\rm d}t}
\|\Xi(t)\|_{L^2}^2
+
(\frac{2}{3}+\mathcal{C})\|\Xi(t)\|_{L^2}^2\\
=
 &
 \int_{\mathbb{R}^3}
 \Big(
 -
 \frac{2}{3}\gamma{\rm div}u\Xi
+
\frac{2\kappa}{3}\gamma
\Delta\theta\Xi
-
ab\Delta j_0\Xi
+
\gamma S^3\Xi
+
b S^4\Xi
 \Big)
 {\rm d}x\\
 \leq
 &
 \frac{1}{3}
 \|\Xi(t)\|_{L^2}^2
 +
 C
 \|\nabla u(t)\|_{L^2}^2
 +
 C
 \|\nabla^2 (\theta,j_0)(t)\|_{L^2}^2
 +
 C
 \|(S^3,S^4)(t)\|_{L^2}^2,
\end{split}
\end{equation}
which implies
\begin{equation}\label{4.39}
\begin{split}
&
\frac{\rm d}{{\rm d}t}
\|\Xi(t)\|_{L^2}^2
+
(\frac{1}{3}+\mathcal{C})\|\Xi(t)\|_{L^2}^2\\
\leq
&
C
 \|\nabla u(t)\|_{L^2}^2
 +
 C
 \|\nabla^2 (\theta,j_0)(t)\|_{L^2}^2
 +
 C
 \|(S^3,S^4)(t)\|_{L^2}^2.
\end{split}
\end{equation}
Multiplying \eqref{4.39} by ${\rm e}^{(\frac{1}{3}+\mathcal{C})t}$ and integrating the resultant inequality with respect to $t$, we get
\begin{equation}\label{4.40}
\begin{split}
\|\Xi(t)\|_{L^2}^2
\leq
&
{\rm e}^{-(\frac{1}{3}+\mathcal{C})t}
\|\Xi(0)\|_{L^2}^2
+
\int_0^t
{\rm e}^{-(\frac{1}{3}+\mathcal{C})(t-\tau)}
 \|\nabla u(\tau)\|_{L^2}^2
{\rm d}\tau\\
&
+
\int_0^t
{\rm e}^{-(\frac{1}{3}+\mathcal{C})(t-\tau)}
\Big(
 \|\nabla^2 (\theta,j_0)(\tau)\|_{L^2}^2
 +
 \|(S^3,S^4)(\tau)\|_{L^2}^2
\Big)
{\rm d}\tau.
\end{split}
\end{equation}
From \eqref{4.40} and Lemma \ref{Lemma 4.2}, we obtain
\begin{equation}
\begin{split}
\|\Xi(t)\|_{L^2}^2
\leq
&
{\rm e}^{-(\frac{1}{3}+\mathcal{C})t}
\|\Xi(0)\|_{L^2}^2
+
C
\int_0^t
{\rm e}^{-(\frac{1}{3}+\mathcal{C})(t-\tau)}
(1+\tau)^{-\frac{5}{4}}
{\rm d}\tau\\
&
+
C
\int_0^t
{\rm e}^{-(\frac{1}{3}+\mathcal{C})(t-\tau)}
(1+\tau)^{-\frac{7}{4}}
{\rm d}\tau\\
\leq
&
C
(1+t)^{-\frac{5}{4}}.
\end{split}
\end{equation}
Similarly, we have
\begin{equation}
\|\nabla^k \Xi(t)\|_{L^2}^2
 \leq
  C
   (1+t)^{-\frac{3}{4}-\frac{k+1}{2}},\ \ \ \ {\rm for}\ k=1,2.
\end{equation}
Thus, the proof of Proposition \ref{Proposition 2.3} is completed. $\hfill{\square}$

\section*{Acknowledgement}
The first author is supported by National Nature Science Foundation of
China 11871341 and 12071152. The second author is supported by National Nature Science Foundation of
China 11571231, 11831003 and Shanghai Science and Technology Innovation Action Plan No. 20JC1413000. The third author is supported by National Nature Science Foundation of
China 11871335 and the SJTU's SMC Projection A.



\section{Appendix}

In the Appendix, we recall some basic facts concerning Littlewood-Paley decomposition, Besov spaces and paraproduct.
Let us first recall the Littlewood-Paley decomposition. For each $j\in \mathbb Z$, set
\begin{equation*}
A_j
 =
  \big\{
  \xi\in\mathbb R^3|2^{j-1}\leq |\xi|\leq 2^{j+1}
  \big\}.
\end{equation*}
The littewoode-paley decomposition asserts the existence of a sequence of functions $\{\varphi_j\}_{j\in\mathbb Z}\subset \mathcal S$ ($\mathcal S$ denotes the usual Schwartz class) such that
\begin{equation*}
{\rm supp}\hat{\varphi}_j\subset A_j,
\ \ \ \
\hat \varphi_j(\xi)
 =
  \hat \varphi_0(2^{-j}\xi)
\ \ {\rm or}\ \
\varphi_j(x)
 =
  2^{3j}\varphi_0(2^jx),
\end{equation*}
and
\begin{equation*}
\sum\limits_{j=-\infty}^\infty
\hat\varphi_j(\xi)
 =
  \left\{
  \begin{array}{lc}
  1,\ \ {\rm if}\ \xi\in\mathbb R^3\setminus \{0\},\\[2mm]
  0,\ \ {\rm if}\ \xi=0.
  \end{array}
  \right.
\end{equation*}
Then the homogeneous Littlewood-Paley decomposition $(\dot{\Delta}_j)_{j\in\mathbb{Z}}$ over $\mathbb{R}^3$ is introduced by setting
\begin{equation*}\label{Def 5.1}
\dot{\Delta}_ju
 :=
  \hat\varphi_j(D)u
  =
  2^{3j}\int_{\mathbb{R}^3}\varphi_0(2^jy)u(x-y){\rm d}y,\ \ \ \ \ j\in \mathbb Z,
\end{equation*}
and
\begin{equation*}
S_ju
 :=
  \sum\limits_{l\leq j-1}
   \dot\Delta_l u,\ \ \ \ \ j\in \mathbb Z.
\end{equation*}
\begin{definition}\label{Def 5.1}
For any $s\in\mathbb{R}$ and $(p,r)\in[1,+\infty]\times[1,+\infty]$, the homogeneous Besov space $\dot B^s_{p,r}(\mathbb{R}^3)$ consists of $f\in \mathcal S^\prime_h=S^\prime/\mathcal P$ satisfying
\begin{equation*}
\|f\|_{\dot B^s_{p,r}(\mathbb{R}^3)}
 :=
  \Big\|2^{sk}\|\dot\Delta_kf\|_{L^p(\mathbb{R}^3)}\Big\|_{l^r(\mathbb{Z})}
   <
    \infty,
\end{equation*}
where $\mathcal S^\prime$ and $\mathcal P$ denote the dual of $\mathcal S$ and the space of polynomials, respectively.
\end{definition}
Note that, for any $f\in \mathcal S^\prime_h$, it can be rewritten as
\begin{equation*}
f=\sum\limits_{k\in \mathbb{Z}}
   \dot\Delta_kf.
\end{equation*}
We define that its long wave part and its short wave part are as follows
\begin{equation*}
f^L:=\sum\limits_{k\leq k_1}
  \dot\Delta_kf
 \ \ \ \ {\rm and}\ \ \ \
f^S:=\sum\limits_{k> k_1}
  \dot\Delta_kf,
\end{equation*}
where the fixed positive integer $k_1$ is defined in \eqref{3.13}.
We also use the following notation
\begin{equation}\label{5.1}
\|f^L\|_{\dot B^s_{p,r}}
 :=
  \Big(
  \sum\limits_{k\leq k_1}
    2^{rsk}
     \|{\dot\Delta}_kf\|_{L^p}^r
  \Big)^\frac{1}{r}
 \ \ \ \ {\rm and}\ \ \ \
 \|f^S\|_{\dot B^s_{p,r}}
 :=
  \Big(
  \sum\limits_{k> k_1}
    2^{rsk}
     \|{\dot\Delta}_kf\|_{L^p}^r
  \Big)^\frac{1}{r}.
\end{equation}

In term of Definition \ref{Def 5.1} and the Plancherel formula, one can find that the homogeneous Sobolev space $\dot H^s$ is a special case of homogeneous Besov spaces as follows, also see \cite{Bahouri-Chemin-Danchin} (page 63) and \cite{Chae-Wan-Wu} (Proposition A.3).
\begin{lemma}\label{Lemma 5.1}
For any $s\in\mathbb{R}$,
\begin{equation*}
\dot H^s \sim \dot B_{2,2}^s.
\end{equation*}
For any $s\in \mathbb R$ and $1<q<\infty$,
\begin{equation*}
\dot B_{q,\min\{q,2\}}^s
\hookrightarrow
\dot W^{s,q}
\hookrightarrow
\dot B_{q,\max\{q,2\}}^s,
\end{equation*}
and
\begin{equation*}
\dot B_{q,\min\{q,2\}}^0
\hookrightarrow
L^q
\hookrightarrow
\dot B_{q,\max\{q,2\}}^0.
\end{equation*}
\end{lemma}

In the following ,we show some useful inequality in Besov space.
\begin{lemma}\label{Lemma 5.2}
For any $s\geq 0$ and $m_1\geq m_2\geq 0$, it holds that
\begin{equation}
C\|\Lambda^{m_2}f^S\|_{\dot B_{2,2}^s}\leq \|\Lambda^{m_1} f^S\|_{\dot B_{2,2}^s}, \ \
\|\Lambda^{m_1} f^L\|_{\dot B_{2,2}^s}\leq C\|\Lambda^{m_2}f^L\|_{\dot B_{2,2}^s},
\end{equation}
and
\begin{equation}\label{5.3}
\|f\|_{\dot B_{2,2}^s}\sim \|f^L\|_{\dot B_{2,2}^s}+\|f^S\|_{\dot B_{2,2}^s}.
\end{equation}
\end{lemma}

\begin{proof}
The proof can be done by using the Plancherel theorem, the Bernstein inequality and the Definition \ref{Def 5.1} and \eqref{5.1}.
\end{proof}

We recall the following estimates, cf. \cite{Chae-Wan-Wu}. Here, we give the proof of Lemma \ref{Lemma 5.3} for the convenience of readers.
\begin{lemma}\label{Lemma 5.3}
Let $v$ be a vector field over $\mathbb{R}^3$ and define the commutator
\begin{equation*}
[\dot \triangle_k,v\cdot \nabla]f=\dot\triangle_k(v\cdot \nabla f)-v\cdot \nabla\dot\triangle_kf.
\end{equation*}
Then the following inequality holds
\begin{equation*}
\begin{split}
\int_{\mathbb{R}^3}[\dot\triangle_k,v\cdot \nabla]f\cdot\dot\triangle_k g
{\rm d}x
\leq
&
C\|\nabla v\|_{L^\infty}\|\dot \Delta_k f\|_{L^2}\|\dot \Delta_k g\|_{L^2}
+
C\|\nabla f\|_{L^\infty}\|\dot \Delta_k v \|_{L^2}\|\dot \Delta_k g\|_{L^2}\\[2mm]
&
+
C\|\nabla v\|_{L^\infty}\|\dot \Delta_k g\|_{L^2}\sum\limits_{l\geq k-1}2^{k-l}\|\dot \Delta_l f\|_{L^2}
.
\end{split}
\end{equation*}
\end{lemma}
\begin{proof}
By using the homogeneous Bony decomposition
\begin{equation*}
\begin{split}
\| [\dot\triangle_k,v\cdot \nabla]f\|_{L^2}
\leq
&
\sum\limits_{|k-l|\leq 2}
\int_{\mathbb{R}^3}
\big(
\dot\Delta_k(S_{l-1} v\cdot \nabla \dot\Delta_l f)-S_{l-1}v\cdot\nabla\dot\Delta_k\dot\Delta_l f
\big)
\cdot
\dot\Delta_kg
{\rm d}x\\
&
+
\sum\limits_{|k-l|\leq 2}
\int_{\mathbb{R}^3}
\big(
\dot\Delta_k(\dot\Delta_l v\cdot \nabla S_{l-1} f)-\dot\Delta_lv\cdot\nabla\dot\Delta_k S_{l-1} f
\big)
\cdot
\dot\Delta_kg
{\rm d}x\\
&
+
\sum\limits_{l\geq k-1}
\int_{\mathbb{R}^3}
\big(
\dot\Delta_k(\dot\Delta_l v\cdot \nabla \tilde\Delta_l f)-\dot\Delta_lv\cdot\nabla\dot\Delta_k \tilde\Delta_l f
\big)
\cdot
\dot\Delta_kg
{\rm d}x\\
:=
&
K_1
+
K_2
+
K_3,
\end{split}
\end{equation*}
with $\tilde\Delta_k=\dot\Delta_{k-1}+\dot\Delta_k+\dot\Delta_{k+1}$. By a standard commutator estimate, the H${\rm \ddot{o}}$lder inequality and Bernstein's inequality, one gets
\begin{equation*}
\begin{split}
K_1
 \leq
  &
   C
    2^{-k}
    \sum\limits_{|k-l|\leq 2}
    \|\nabla S_{l-1} v\|_{L^\infty}
     \|\nabla \dot\Delta_l f\|_{L^2}
      \|\dot\Delta_k g\|_{L^2}\\
  \leq
   &
    C
      \|\nabla v\|_{L^\infty}
       \sum\limits_{|k-l|\leq 2}
     \|\dot\Delta_l f\|_{L^2}
      \|\dot\Delta_k g\|_{L^2}.
\end{split}
\end{equation*}
Similarly, one obtains
\begin{equation*}
\begin{split}
K_2
  \leq
   &
    C
      \|\nabla f\|_{L^\infty}
       \sum\limits_{|k-l|\leq 2}
     \|\dot\Delta_l v\|_{L^2}
      \|\dot\Delta_k g\|_{L^2},
\end{split}
\end{equation*}
\begin{equation*}
\begin{split}
K_3
  \leq
   &
    C
       \sum\limits_{l\geq k-1}
       2^{k-l}
      \|\nabla \dot\Delta_l v\|_{L^\infty}
     \|\dot\Delta_l f\|_{L^2}
      \|\dot\Delta_k g\|_{L^2}\\
  \leq
   &
    C
     \|\nabla v\|_{L^\infty}
       \sum\limits_{l\geq k-1}
       2^{k-l}
     \|\dot\Delta_l f\|_{L^2}
      \|\dot\Delta_k g\|_{L^2}.
\end{split}
\end{equation*}
Since the summation over $l$ for fixed $k$ above consists of only a finite number of terms and the norm generated by each term is a multiple of that generated by the typical term, it suffices to keep the typical term with $l=k$ and ignore the summation. This would help keep the presentation concise.
Therefore, the proof of Lemma \ref{Lemma 5.3} is completed.
\end{proof}

We list some important estimates in Sobolev space, which can be found in \cite{Chen-Ding-Wang,Majda-Bertozzi,Stein,Wang}.

\begin{lemma}\label{Lemma 5.4}
Let $f\in H^2(\mathbb{R}^3)$. Then

{\rm (i)} $\|f\|_{L^\infty}\leq C\|\nabla f\|^{1/2}\|\nabla f\|_{H^1}^{1/2}\leq C\|\nabla f\|_{H^1}$;

{\rm (ii)} $\|f\|_{L^6}\leq C\|\nabla f\|$;

{\rm (iii)} $\|f\|_{L^q}\leq C\|f\|_{H^1},\ \ 2\leq q\leq 6$.
\end{lemma}

\begin{lemma}\label{Lemma 5.5}
Let $m\geq 1$ be an integer, then we have
\begin{equation}
\|\nabla^m(fg)\|_{L^p}
\leq
  C\|f\|_{L^{p_1}}
   \|\nabla^m g\|_{L^{p_2}}
    +
     C\|\nabla^m f\|_{L^{p_3}}
      \|g\|_{L^{p_4}},
\end{equation}
where $1\leq p_i\leq +\infty,\ (i=1,2,3,4)$ and
\begin{equation*}
\frac{1}{p}
 =
  \frac{1}{p_1}
   +
    \frac{1}{p_2}
     =
      \frac{1}{p_3}
       +
        \frac{1}{p_4}.
\end{equation*}
\end{lemma}

\vspace{1mm}

\begin{lemma}\label{Lemma 5.6}
Assume that $\|\phi\|_{L^\infty}\leq 1$. Let $f(\phi)$ be a smooth function of $\phi$ with bounded derivatives of any order, then for any integer $m\geq 1$ and $1\leq p\leq +\infty$, we have
\begin{equation*}
\|\nabla^m f(\phi)\|_{L^p}
 \leq
  C\|\nabla^m \phi\|_{L^p}.
\end{equation*}
\end{lemma}

\vspace{1mm}

\begin{lemma}\label{Lemma 5.7}
Let $0<l<3$, $1<p<q<\infty$, $\frac{1}{q}+\frac{l}{3}=\frac{1}{p}$, then
\begin{equation}
\|\Lambda^{-l}f\|_{L^q}
 \lesssim
  \|f\|_{L^p}.
\end{equation}
\end{lemma}


\bibliographystyle{plain}

\end{document}